\documentclass[11pt]{article}

\usepackage[latin1]{inputenc}
\usepackage{tikz}
\usepackage{bbm}
\usepackage{amsmath, amsthm, amsfonts, amssymb, mathrsfs}
\usetikzlibrary{matrix,arrows}
\usetikzlibrary{arrows}
\usepackage{arydshln}
\usepackage{multirow}
\usepackage{subfig}

\DeclareMathOperator{\tr}{tr}
\DeclareMathOperator{\Argmin}{Argmin}
\DeclareMathOperator{\Argmax}{Argmax}
\DeclareMathOperator{\argmin}{argmin}
\DeclareMathOperator{\argmax}{argmax}
\DeclareMathOperator{\dual}{dual}
\DeclareMathOperator{\cone}{cone}
\DeclareMathOperator{\val}{val}
\DeclareMathOperator{\sol}{sol}
\DeclareMathOperator{\Prob}{Prob}

\DeclareMathOperator{\spa}{span}
\DeclareMathOperator{\Skew}{Skew}

\DeclareMathOperator{\ndet}{ndet}
\DeclareMathOperator{\inter}{int}
\DeclareMathOperator{\relint}{relint}

\DeclareMathOperator{\Her}{Her}

\DeclareMathOperator{\Gr}{Gr}
\DeclareMathOperator{\diag}{diag}
\DeclareMathOperator{\rvol}{rvol}
\DeclareMathOperator{\vol}{vol}
\DeclareMathOperator{\im}{im}

\DeclareMathOperator{\patt}{patt}
\DeclareMathOperator{\rk}{rk}
\DeclareMathOperator{\Sol}{Sol}
\DeclareMathOperator{\SDP}{SDP}
\DeclareMathOperator{\hSDP}{hSDP}
\DeclareMathOperator{\CP}{CP}
\DeclareMathOperator{\hCP}{hCP}

\DeclareMathOperator{\codim}{codim}
\DeclareMathOperator{\afftmp}{aff}
\newcommand{\aff}{{\afftmp}}
\newcommand{\pa}{\patt}

\newcommand{\be}{\beta}
\newcommand{\de}{\delta}
\newcommand{\seq}{\succeq}
\newcommand{\IE}{\mathbb{E}}
\newcommand{\IF}{\mathbb{F}}
\newcommand{\IN}{\mathbb{N}}
\newcommand{\IR}{\mathbb{R}}
\newcommand{\IC}{\mathbb{C}}
\newcommand{\IH}{\mathbb{H}}
\newcommand{\IZ}{\mathbb{Z}}
\newcommand{\bi}{\mathbf{i}}
\newcommand{\bj}{\mathbf{j}}
\newcommand{\bk}{\mathbf{k}}

\newcommand{\mB}{\mathscr{B}}
\newcommand{\hmB}{\hat{\mathscr{B}}}

\newcommand{\mC}{\mathscr{C}}
\newcommand{\E}{\mathcal{E}}
\newcommand{\mO}{\mathcal{O}}
\newcommand{\mI}{\mathcal{I}}

\newcommand{\Herbn}{\Her_{\be,n}}
\newcommand{\vp}{\varphi}
\newcommand{\B}{\mathcal{B}}
\newcommand{\C}{\mathcal{C}}
\newcommand{\F}{\mathcal{F}}
\newcommand{\M}{\mathcal{M}}
\newcommand{\N}{\mathcal{N}}

\newcommand{\W}{\mathcal{W}}

\newcommand{\mcL}{\mathcal{L}}
\newcommand{\T}{\mathcal{T}}
\newcommand{\HG}{{ }_2F_1}
\newcommand{\Cbn}{\C_{\be,n}}
\newcommand{\Wbnr}{\W_{\be,n,r}}

\newcommand{\GbE}{\text{G}\beta\text{E}}

\newcommand{\vw}{
\begin{tikzpicture}[scale=0.35, >=stealth]\draw[->] (0,0) -- (45:1) node[right]{$w$}; \draw[->] (0,0) -- (135:1) node[left]{$v$};\end{tikzpicture}
}
\newcommand{\wv}{
\begin{tikzpicture}[scale=0.35, >=stealth]\draw[->] (0,0) -- (45:1) node[right]{$v$}; \draw[->] (0,0) -- (135:1) node[left]{$w$};\end{tikzpicture}
}

\newcommand{\ol}[1]{\overline{#1}}
\newcommand{\wh}[1]{\tilde{#1}}

\newtheorem{thm}{Theorem}[section]
\newtheorem{corollary}[thm]{Corollary}
\newtheorem{lemma}[thm]{Lemma}
\newtheorem{proposition}[thm]{Proposition}

\theoremstyle{definition}
\newtheorem{definition}[thm]{Definition}
\newtheorem{remark}[thm]{Remark}

\usepackage[pdftex, pdfpagelabels]{hyperref}
\hypersetup{
  plainpages=false,
  colorlinks,
  citecolor=black,
  filecolor=black,
  linkcolor=black,
  urlcolor=black,
  pdfpagemode = UseNone,
  pdfstartview={XYZ null null 0.75},
}

\begin{document}

\title{Intrinsic volumes of symmetric cones}
\author{Dennis Amelunxen$^\ast$ and Peter B\"urgisser
\thanks{Institute of Mathematics, University of Paderborn, Germany. 
        Partially supported by DFG grants BU \mbox{1371/2-2} and AM \mbox{386/1-1}.}
     \\ University of Paderborn
     \\ \{damelunx,pbuerg\}@math.upb.de
        }
\date{\today}
\maketitle

\begin{abstract}
We compute the intrinsic volumes of the cone of positive semidefinite matrices over the real numbers, over the complex numbers, and over the quaternions, in terms of integrals related to Mehta's integral. Several applications for the probabilistic analysis of semidefinite programming are given.
\end{abstract}

\smallskip

\noindent{\bf AMS subject classifications:} 15B48, 52A55, 53C65, 60D05, 90C22

\smallskip

\noindent{\bf Key words:} intrinsic volumes, symmetric cones, Mehta's integral, semidefinite programming

%



\section{Introduction}

The classification of symmetric cones, also known as self-scaled cones, i.e., closed convex cones, which are self-dual, and whose automorphism group acts transitively on the interior, is a well-known result. It says that every symmetric cone is a direct product of the following basic families of symmetric cones:
\begin{itemize}
  \item the Lorentz cones, which have the form $\{x\in\IR^n\mid x_n\geq (x_1^2+\ldots+x_{n-1}^2)^{1/2}\}$,
  \item the cones of positive semidefinite matrices over the real numbers, the complex numbers, or the quaternions,
  \item the single (exceptional, $27$-dimensional) cone of $3\times 3$ positive semidefinite matrices over the octonions.
\end{itemize}
This result follows from the theory of Jordan algebras, which is intimately related to the theory of symmetric cones, cf.~\cite{FK:94}.

On the other hand, self-scaled cones form the basis of interior-point methods in convex optimization. This has been observed in the mid~'90s, cf.~\cite{neto:97,neto:98,Gue:96,Fayb:97}, cf.~also the book~\cite{Ren:01} and the survey article~\cite{HG:02}. A detailed understanding of these cones, in particular of its statistical properties, is thus of fundamental importance. By `statistical properties' we mean basic probabilities like the probability that a random convex program is feasible, etc. In the case of linear programming (LP) these fundamental statistical properties are well understood. For example, the probability that a random LP is feasible has been computed by Wendel in~1962~\cite{wend:62}, and further elementary probabilities can be deduced from this, cf.~\cite[Thm.~4]{chcu:04}. In fact, the probability that a random convex program is feasible can be expressed in terms of the \emph{intrinsic volumes} of the reference cone. The intrinsic volumes of a closed convex cone form a discrete probability distribution that to some extent captures the statistics of the cone.
In the case of linear programming the reference cone is the positive orthant, and its intrinsic volumes are given by the symmetric binomial distribution, cf.~Section~\ref{sec:intrvol-polyhedr}. As for second-order cone programming, the reference cone is the Lorentz cone, whose intrinsic volumes are also, basically, given by the symmetric binomial distribution (see Remark~\ref{rem:completeness} for the details).


We give in this paper, apparently for the first time, an explicit formula for the intrinsic volumes of the cone of positive semidefinite matrices over the real numbers, over the complex numbers, and over the quaternions. The resulting formulas involve integrals that are related to Mehta's integral. It remains to give closed formulas and/or derive asymptotics for these integrals, but we see this work as the first step for understanding the intrinsic volumes of the symmetric cones, and thus for understanding fundamental statistical properties of semidefinite programming.
To illustrate the significance of the intrinsic volumes (and of its local versions, the~\emph{curvature measures}), we will give several corollaries, which describe some interesting probabilities about semidefinite programming (SDP) in terms of these integrals. In particular, we obtain a closed formula for the probability that the solution of a random SDP has a certain rank. To the best of our knowledge, this is so far the first result making concrete statements about the above-mentioned probability, a question which is by now at least 15 years old, cf.~\cite{AHO:97}.

Another interesting aspect, which deserves further investigation, is the observation that there seems to be a connection between the curvature measures of the cone of positive semidefinite matrices and the \emph{algebraic degree of semidefinite programming}. This degree has been defined in~\cite{NRS:09}, cf.~also~\cite{BR:09}, and some remarkable parallels to the curvature measures can be established, cf.~Remark~\ref{rem:connect-algdeg}. In fact, the authors also describe in~\cite[Sec.~3]{NRS:09} an experiment to empirically analyze the rank of the solution of a random semidefinite program. But their (experimental) results are not comparable to our 
formulas, as they choose a different distribution to specify a random semidefinite program.

The organization of the paper is as follows. In Section~\ref{sec:appl-CP} we describe the connection between the specific curvature measures, for which we will give explicit formulas in Section~\ref{sec:form-Phi_j(be,n,r)}, and some fundamental statistics in semidefinite programming. Although we defer the formal definition of the curvature measures and the intrinsic volumes to Section~\ref{sec:spher-intr-vol}, we give the SDP application at this early stage to motivate the subsequent sometimes technical sections. Section~\ref{sec:appl-kin-form} is devoted to the applications of the kinematic formula in the context of convex programming. This section is independent from the rest of the paper, but it shows the usefulness of the kinematic formula and further motivates the computation of the intrinsic volumes of the cone of semidefinite matrices. In Section~\ref{sec:main_res} we give the main result of this paper, which are the formulas for the intrinsic volumes/curvature measures of the cone of positive semidefinite matrices. Section~\ref{sec:proof-main-res} is devoted to the proof of the main result.


Since we need in Section~\ref{sec:appl-kin-form} a specific form of the kinematic formula, which is not easy to trace in the literature, we will describe in the appendix the additional concept of \emph{support measures}. The most general form of the kinematic formula in the spherical setting is stated in terms of these measures. We will describe how the specific kinematic formula that we use can be derived from this general result.

\subsection{Applications in convex programming}\label{sec:appl-CP}

We consider the following forms of convex programming. Let $\E$ be a finite-dimensional Euclidean space with inner product $\langle .,.\rangle\colon\E\times\E\to\IR$. Furthermore, let $C\subseteq\E$ be a closed convex cone, i.e., a closed set that satisfies $0\in C$ and $\lambda x+\mu y\in C$ for all $x,y\in C$ and $\lambda,\mu\geq0$. The classical convex programming problem (with \emph{reference cone}~$C$) has the inputs $a_1,\ldots,a_m,z\in \E$ and $b_1,\ldots,b_m\in\IR$, and consists of the task
\begin{align}\label{eq:CP}\tag{$\CP$}
   \text{maximize }\; & \langle z, x\rangle
\\ \text{subject to }\; & \langle a_i, x\rangle =b_i \,,\; i=1,\ldots,m, \nonumber
\\ & x\in C , \nonumber
\end{align}
which is to be solved in~$x\in\E$. We also define a homogeneous version, which is easier to analyze. This has only the inputs $a_1,\ldots,a_m,z\in\E$, and is again to be solved in~$x\in\E$
\begin{align}\label{eq:hCP}\tag{$\hCP$}
   \text{maximize }\; & \langle z, x\rangle
\\ \text{subject to }\; & \langle a_i, x\rangle =0 \,,\; i=1,\ldots,m, \nonumber
\\ & x\in C \,,\; \|x\|\leq 1 . \nonumber
\end{align}

The space~$\E$ is endowed with the \emph{(standard) normal distribution}~$\N(\E)$. Choosing an orthonormal basis of~$\E$ so that we have an isometry $\vp\colon\E\to\IR^d$, $d=\dim\E$, a random element $x\in\E$ is normal distributed iff the components of $\vp(x)$ are i.i.d.~standard normal, i.e., in~$\N(0,1)$. We say that an instance of~\eqref{eq:CP} is a \emph{(normal) random program} if the inputs $a_1,\ldots,a_m,z$ are i.i.d.~in $\N(\E)$ and the inputs $b_1,\ldots,b_m$ are i.i.d.~in $\N(0,1)$. Analogously, we speak of a random instance of~\eqref{eq:hCP} if the inputs $a_1,\ldots,a_m,z$ are i.i.d.~in $\N(\E)$.

For the analysis of the problems~\eqref{eq:CP} and~\eqref{eq:hCP} we use the following notation: We denote the \emph{feasible set} of~\eqref{eq:CP} and~\eqref{eq:hCP} by
\begin{align*}
   \F(\CP) & := \{ x\in C \mid \forall i: \langle a_i, x\rangle=b_i \} , & \F(\hCP) & := \{ x\in C \mid \forall i: \langle a_i, x\rangle =0 , \|x\|\leq 1\} .
\intertext{By $\val$ we denote the \emph{value} of~\eqref{eq:CP} or~\eqref{eq:hCP},}
   \val(\CP) & := \sup\{ \langle z, x\rangle \mid x\in\F(\CP) \} , & \val(\hCP) & := \sup\{ \langle z, x\rangle \mid x\in\F(\hCP) \} .
\intertext{Finally, we denote by $\Sol$ the \emph{solution set} of~\eqref{eq:CP} or~\eqref{eq:hCP},}
   \Sol(\CP) & := \{ x\in\F(\CP) \mid \langle z, x\rangle = \val(\CP)\} , & \Sol(\hCP) & := \{ x\in\F(\hCP) \mid \langle z, x\rangle = \val(\hCP)\} .
\end{align*}
Note that $\val(\hCP)$ is in fact a maximum, as the set $\F(\hCP)$ is always compact and contains the origin. In general, this is not the case for the affine version~\eqref{eq:CP}. The feasible set $\F(\CP)$ may be unbounded, and the value $\val(\CP)$ may be~$\infty$ in which case we say that~\eqref{eq:CP} is \emph{unbounded}. Also, the feasible set $\F(\CP)$ may be empty, so that $\val(\CP)=\sup \emptyset$ which is~$-\infty$ by definition. In this case we say that~\eqref{eq:CP} is \emph{infeasible}. If the solution set~$\Sol(\CP)$ only consists of a single element, then we denote this by~$\sol(\CP)$. So writing $x_0=\sol(\CP)$ means that $\Sol(\CP)=\{x_0\}$. We use a similar convention for~$\Sol(\hCP)$. Well-known results from convex geometry, see for example~\cite[Thm.~2.2.9]{Schn:book}, imply that almost surely $\Sol(\hCP)$ and $\Sol(\CP)$ are either empty or consist of single elements.

For random instances of~\eqref{eq:CP} and~\eqref{eq:hCP} we can express certain statistics in terms of the \emph{intrinsic volumes} and the \emph{curvature measures} of the reference cone. The intrinsic volumes of the closed convex cone~$C\subseteq\E$ are nonnegative numbers $V_0(C),\ldots,V_d(C)$, $d=\dim\E$, which add up to one, $V_0(C)+\ldots+V_d(C)=1$. To describe the curvature measures, we make the following definition (cf.~Section~\ref{sec:intrvol-polyhedr} for a discussion of this concept).

\begin{definition}
Let $\E$ be a finite-dimensional euclidean space, and let 
$\mB(\E)$ denote the Borel $\sigma$-algebra on $\E$. Then we call
\begin{equation}\label{eq:def-hat(B)(E)}
  \hmB(\E) := \{ M\in\mB(\E)\mid \forall \lambda>0: \lambda M=M\}
\end{equation}
the \emph{conic (Borel) $\sigma$-algebra on $\E$}.
\end{definition}

It is easily seen that $\hmB(\E)$ indeed satisfies the axioms of a $\sigma$-algebra. The curvature measures of~$C$ are measures $\Phi_0(C,.),\ldots,\Phi_d(C,.)\colon \hmB(\E)\to\IR_+$, which satisfy $\Phi_j(C,\E)=V_j(C)$. So the curvature measures localize the intrinsic volumes. We give an illustrative characterization of the intrinsic volumes and the curvature measures of polyhedral cones in Section~\ref{sec:intrvol-polyhedr}, and we shall provide the definition for the general case in Section~\ref{sec:spher-intr-vol}. In Section~\ref{sec:supp-meas} in the appendix we will also explain the \emph{support measures}, which, among other things, will further justify our definition of $\hmB(\E)$.

The following two theorems describe the statistics of~\eqref{eq:hCP} and~\eqref{eq:CP}.

\begin{thm}\label{thm:hCP}
We have the following probabilities for random instances of~\eqref{eq:hCP}
\begin{equation}\label{eq:hCP-probs1}
  \Prob\big[\sol(\hCP)=0\big] = \sum_{j=0}^m V_j(C) \;,\quad \Prob\big[\sol(\hCP)\in M\big] = \sum_{j=m+1}^d \Phi_j(C,M) \; ,
\end{equation}
where $M\in\hmB(\E)$ with $0\not\in M$. Furthermore, if $C$ is not a linear subspace, then
\begin{equation}\label{eq:hCP-probs2}
  \Prob\big[\F(\hCP)=\{0\}\big] = 2\cdot \hspace{-2mm}\sum_{\substack{j=0\\j\equiv m-1 \bmod2}}^{m-1} V_j(C) \; .
\end{equation}
\end{thm}

\begin{thm}\label{thm:CP}
We have the following probabilities for random instances of~\eqref{eq:CP}
\begin{equation}\label{eq:CP-probs1}
   \Prob\big[\CP \text{ infeasible}\big] = \sum_{j=0}^{m-1} V_j(C) \;,\quad
   \Prob\big[\CP \text{ unbounded}\big] = \sum_{j=m+1}^d V_j(C) \; .
\end{equation}
Furthermore, for $M\in\hmB(\E)$ we have
\begin{equation}\label{eq:CP-probs2}
   \Prob\big[\sol(\CP)\in M\big]  = \Phi_m(C,M) \; ,
\end{equation}
and $\Prob\big[\sol(\CP)\in M \wedge \val(\CP)>0\big] = \Prob\big[\sol(\CP)\in M \wedge \val(\CP)<0\big]$.
\end{thm}

\begin{remark}
The intrinsic volumes of the positive orthant~$\IR_+^d$ are given by the symmetric binomial distribution $V_j(\IR_+^d)=\binom{d}{j}/2^d$, cf.~Remark~\ref{rem:intrvol-LP}. Plugging in these values in Theorem~\ref{thm:CP} yields the statistics of linear programming, which are repeatedly computed,
cf.~\cite{AB:81,S:83,chcu:04}.
\end{remark}

We finish this introduction with a discussion of the special case of \emph{semidefinite programming}. For this we need to set up some notation first. Throughout the paper we use the parameter $\be\in\{1,2,4\}$ to indicate if we are working over the real numbers~$\IR$, over the complex numbers~$\IC$, or over the quaternions~$\IH$. In particular, we denote the ground (skew) field by $\IF_\be$, i.e.,
  \[ \IF_1:=\IR \;,\quad \IF_2:=\IC \;,\quad \IF_4:=\IH \; . \]
We consider these (skew) fields with the natural identifications
  \[ \IR\subset\IC=\IR[\bi]\subset\IH=\IC[\bj,\bk]=\IR[\bi,\bj,\bk] \; , \]
where $\bi^2=\bj^2=\bk^2=-1$, and with the well-known quaternion multiplication rule found by Hamilton in~1843. Recall that we have the (linear) conjugation map $\bar{.}\colon\IH\to\IH$, given by $\bar{1}=1$, $\bar{\bi}=-\bi$, $\bar{\bj}=-\bj$, $\bar{\bk}=-\bk$, and its restriction to $\IC$. Furthermore, we call $\Re\colon\IF_\be\to\IR$, $z\mapsto \frac{z+\bar{z}}{2}$ the canonical projection on~$\IR$.

The set of $(n\times n)$-Hermitian matrices over $\IF_\be$ shall be denoted by
  \[ \Herbn := \{ A\in\IF_\be^{n\times n}\mid A^\dagger = A \} \; , \]
where $A^\dagger=(\bar{a}_{ji})$ for $A=(a_{ij})$. This is a real vector space of dimension
\begin{equation}\label{eq:def-d_(b,n)}
  d_{\be,n} := \dim \Herbn = n+\be\cdot \tbinom{n}{2} \; .
\end{equation}
Throughout this paper we regard 
$\Herbn$ as a euclidean vector space with the inner product given by 
$A\bullet B := \Re(\tr(A^\dagger B))$, where 
$A,B\in\Herbn$, and $\tr(A):=a_{11}+\ldots+a_{nn}$ denotes the trace. The normal distribution in~$\Herbn$ with respect to this inner product is called the \emph{Gaussian Orthogonal/Unitary/Symplectic Ensemble} (GOE/GUE/GSE), according to $\be=1,2,4$, respectively. We use the short notation~$\GbE$ for this distribution.

For $A\in\Herbn$ and $x\in\IF_\be^n$ the product $x^\dagger A x$ lies in~$\IR$, and an element $A\in\Herbn$ is called \emph{positive semidefinite} iff $x^\dagger A x\geq0$ for all $x\in\IF_\be^n$. The set of all positive semidefinite elements in $\Herbn$ is a closed convex cone, the \emph{cone of positive semidefinite matrices over~$\IF_\be$}, which we denote by
\begin{equation}\label{eq:def-C_(b,n)}
  \C_{\be,n} = \{ A\in \Herbn\mid \forall x\in \IF_\be^n: x^\dagger A x\geq0 \} \; .
\end{equation}
The cone $\C_{\be,n}$ has a natural decomposition according to the rank of the matrices (cf.~\cite{Z:97} for the quaternion case)
\begin{equation}\label{eq:decomp-C-rank}
  \C_{\be,n} = \bigcup_{r=0}^n \; \W_{\be,n,r} \;,\qquad \W_{\be,n,r} := \{A\in \C_{\be,n}\mid \rk A=r\} \; .
\end{equation}
For the $j$th curvature measure of $\Cbn$ evaluated at the set of its rank~$r$ matrices we write
\begin{equation}\label{eq:def-Phi_j(b,n,r)}
  \Phi_j(\be,n,r) := \Phi_j(\Cbn,\Wbnr) \; .
\end{equation}
The decomposition~\eqref{eq:decomp-C-rank} of the cone $\Cbn$ into the rank~$r$ strata yields the formula
\begin{equation}\label{eq:decomp-V_j(Cbn)}
  V_j(\Cbn) = \sum_{r=0}^n \Phi_j(\be,n,r) \;,\quad j=0,\ldots,d_{\be,n} \; .
\end{equation}

Semidefinite programming is now the following specialization of convex programming
  \[ \E=\Herbn \,,\quad C = \Cbn \,,\quad \langle x, y\rangle \;\hat{=}\; X\bullet Y \,,\quad \N(\E) \;\hat{=}\; \GbE \; , \]
and we obtain the following programs
\\\begin{minipage}{0.45\textwidth}
\begin{align}\label{eq:SDP}\tag{$\SDP_\be$}
   \text{max.~}\; & Z\bullet X
\\ \text{s.t.~}\; & A_i\bullet X=b_i 
\nonumber
\\ & X\seq 0 , \nonumber
\end{align}
\end{minipage}
\begin{minipage}{0.45\textwidth}
\begin{align}\label{eq:hSDP}\tag{$\hSDP_\be$}
   \text{max.~}\; & Z\bullet X
\\ \text{s.t.~}\; & A_i\bullet X=0 
\nonumber
\\ & X\seq 0 \,,\; \|X\|\leq 1 , \nonumber
\end{align}
\end{minipage}
\\[3mm]where $X\seq0$ is the usual notation for $X\in\Cbn$.

Specializing Theorem~\ref{thm:CP} yields the following corollary for semidefinite programming.


\begin{corollary}\label{cor:SDP}
We have the following probabilities for random instances of~\eqref{eq:SDP}
\begin{equation}
   \Prob\big[\SDP_\be \text{ infeasible}\big] = \sum_{j=0}^{m-1} V_j(\Cbn) \;,\quad
   \Prob\big[\SDP_\be \text{ unbounded}\big] = \sum_{j=m+1}^d V_j(\Cbn) \; .
\end{equation}
Furthermore, for $0\leq r\leq n$ we have
\begin{equation}
   \Prob\big[\rk(\sol(\SDP_\be))=r\big]  = \Phi_m(\be,n,r) \; .
\end{equation}
\end{corollary}

See Section~\ref{sec:form-Phi_j(be,n,r)} for explicit formulas for~$V_j(\Cbn)$ and $\Phi_j(\be,n,r)$. For $\be=4$, $n=3$, $m=6$ the probabilities from Corollary~\ref{cor:SDP} are shown in Figure~\ref{fig:decomp-intr-vol}.

\begin{figure}
  \centerline{\begin{tikzpicture}[xscale=0.8, yscale=20]
  \draw (-1,0) -- (16,0);
  \foreach \x in {0,1,...,15} \draw (\x,0) -- (\x,-0.005) node[below]{$V_{\x}$};
  \draw[thick] (-0.5,0) rectangle (0.5,0.0003286279277);
  \draw[thick] (0.5,0) rectangle (1.5,0.002109727388);
  \draw[thick] (1.5,0) rectangle (2.5,0.007957747152);
  \draw[thick] (2.5,0) rectangle (3.5,0.02238263631);
  \draw[thick] (3.5,0) rectangle (4.5,0.05039906529);
  \draw[thick] (4.5,0) rectangle (5.5,0.09375000000);
  \draw[thick] (5.5,0) rectangle (6.5,0.1438585711);
  \draw[thick] (6.5,0) rectangle (7.5,0.1792136249);
  \draw[thick] (7.5,0) rectangle (8.5,0.1792136249);
  \draw[thick] (8.5,0) rectangle (9.5,0.1438585711);
  \draw[thick] (9.5,0) rectangle (10.5,0.09375000000);
  \draw[thick] (10.5,0) rectangle (11.5,0.05039906529);
  \draw[thick] (11.5,0) rectangle (12.5,0.02238263631);
  \draw[thick] (12.5,0) rectangle (13.5,0.007957747152);
  \draw[thick] (13.5,0) rectangle (14.5,0.002109727388);
  \draw[thick] (14.5,0) rectangle (15.5,0.0003286279277);
  \draw (5.5,0.1273239544) -- node[above]{$\scriptstyle2$} ++(1,0);
  \draw (6.5,0.1155516477) -- node[above]{$\scriptstyle2$} ++(1,0);
  \draw (7.5,0.06366197722) -- node[above]{$\scriptstyle2$} ++(1,0);
  \draw (8.5,0.01653461654) -- node[above]{$\scriptstyle2$} ++(1,0);
  \path (0,0) node[above]{$\scriptstyle0$}
        (1,0) node[above]{$\scriptstyle1$}
        (2,0) node[above]{$\scriptstyle1$}
        (3,0) node[above]{$\scriptstyle1$}
        (4,0) node[above]{$\scriptstyle1$}
        (5,0) node[above]{$\scriptstyle1$}
        (6,0) node[above]{$\scriptstyle1$}
        (7,0) node[above]{$\scriptstyle1$}
        (8,0) node[above]{$\scriptstyle1$}
        (9,0) node[above]{$\scriptstyle1$}
        (10,0) node[above]{$\scriptstyle2$}
        (11,0) node[above]{$\scriptstyle2$}
        (12,0) node[above]{$\scriptstyle2$}
        (13,0) node[above]{$\scriptstyle2$}
        (14,0) node[above]{$\scriptstyle2$}
        (15,0) node[above]{$\scriptstyle3$};
  \draw[dashed] (5,0.21) node[below=3mm, left=-3.5mm]{$\Prob[\SDP_4 \text{ infeasible}]$} -- (5.5,0.21) -- (5.5,-0.03) -- (5,-0.03);
  \draw[dashed] (7,0.21) node[below=3mm, right=-3.5mm]{$\Prob[\SDP_4 \text{ unbounded}]$}-- (6.5,0.21) -- (6.5,-0.03) -- (7,-0.03);
  \draw[->,>=stealth] (4.5,0.14) node[left]{$\Prob[\rk(\sol(\SDP_4))=2]$} -- (5.7,0.135);
  \draw[->,>=stealth] (4.5,0.115) node[left]{$\Prob[\rk(\sol(\SDP_4))=1]$} -- (5.7,0.1);
  \end{tikzpicture}}
  \caption{The intrinsic volumes of $\C_{4,3}$ and their decompositions in curvature measures. The small numbers indicate the contributions of the ranks. The probabilities from Corollary~\ref{cor:SDP} for $m=6$ are also indicated.}
  \label{fig:decomp-intr-vol}
\end{figure}
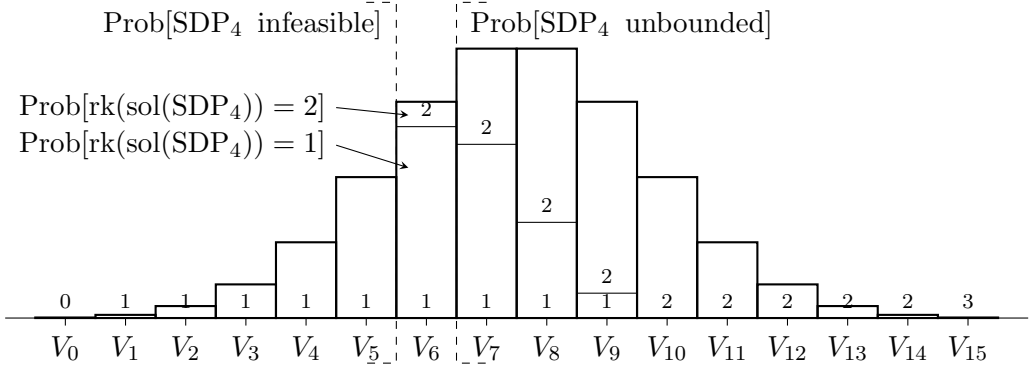

\section{Applications of the kinematic formula}\label{sec:appl-kin-form}

The goal here is to provide the proofs for Theorem~\ref{thm:hCP} and Theorem~\ref{thm:CP}. In Section~\ref{sec:intrvol-polyhedr} we first introduce the notion of intrinsic volumes and curvature measures for polyhedral cones; the case of general closed convex cones is deferred to Section~\ref{sec:spher-intr-vol}. The kinematic formula will be presented in Section~\ref{sec:kinem-form}. In Section~\ref{sec:hCP} and Section~\ref{sec:CP} we prove Theorem~\ref{thm:hCP} and Theorem~\ref{thm:CP}, respectively, by means of the kinematic formula.

\subsection{Intrinsic volumes of polyhedral cones}\label{sec:intrvol-polyhedr}

The \emph{intrinsic volumes} of closed convex cones are usually defined in the spherical setting, which is obtained by intersecting the cone with the unit sphere. We will give the definition of the spherical intrinsic volumes, which is a bit technical, in Section~\ref{sec:spher-intr-vol}. However, we would like to mention at this point the close relationship between the conic $\sigma$-algebra $\hmB(\E)$, that we defined in~\eqref{eq:def-hat(B)(E)}, and the Borel algebra on the unit sphere $S(\E)=\{x\in\E\mid \|x\|=1\}$. Namely, we have the decomposition $\hmB(\E)=\hmB_0(\E)\;\dot{\cup}\;\hmB_\emptyset(\E)$, where $\hmB_0(\E) := \{ M\in \hmB(\E)\mid 0\in M\}$ and $\hmB_\emptyset(\E) := \{ M\in \hmB(\E)\mid 0\not\in M\}$. A moment of thought reveals that the mappings
  \[ \hmB_0(\E)\to \mB(S(\E)) \,,\; M\mapsto M\cap S(\E) \;,\qquad \hmB_\emptyset(\E)\to \mB(S(\E)) \,,\; M\mapsto M\cap S(\E) \]
are bijections, i.e., we may identify both $\hmB_0(\E)$ and $\hmB_\emptyset(\E)$ with the Borel algebra $\mB(S(\E))$. So it might seem that our definition of~$\hmB(\E)$ is superfluous, or overly pedantic. But in fact, the use of~$\hmB(\E)$ is not only convenient, as we will see in the course of this section, but also valuable in the context of the support measures, that we will describe in Section~\ref{sec:supp-meas}. In the following paragraphs we will give an illustrative characterization of the curvature measures and the intrinsic volumes for polyhedral cones, i.e., intersections of finitely many closed half-spaces.

If $C\subseteq\IR^d$ is a closed convex cone, we denote by $\breve{C}:=\{x\in\IR^d\mid \forall y\in C:\langle x,y\rangle\leq0\}$ the \emph{dual cone} of~$C$ in~$\IR^d$. (Occasionally, we will also use the notation $\dual(C):=\breve{C}$, if this is more convenient.) The most important cones used in convex programming are \emph{self-dual}, i.e., $\breve{C}=-C$; it is well-known, cf.~for example~\cite[\S II.12]{Barv}, that~$\Cbn$ is self-dual, i.e., $\dual(\Cbn)=-\Cbn$.

A \emph{supporting hyperplane}~$H$ of~$C$ is a hyperplane such that~$C$ lies in one of the closed half-spaces bounded by~$H$. The intersection~$H\cap C$ is called a \emph{face}\footnote{Some authors differentiate between faces and exposed faces, cf.~for example~\cite{Schn:book}. We do not make this distinction, as for those cones in which we are interested both notions coincide.} of~$C$.

A \emph{polyhedral cone} $C\subseteq\IR^d$ is the intersection of finitely many closed half-spaces bounded by linear hyperplanes. The boundary of the cone~$C$ decomposes in the disjoint union of the relative interiors of its faces. More precisely, we have
  \[ C=\dot{\bigcup}_{F\in\F} \, F \;,\qquad \F:=\{\relint(C\cap v^\bot)\mid v\in \breve{C}\} \; , \]
where $v^\bot := \{x\in\IR^d\mid \langle x,v\rangle=0\}$. Let $\F_j:=\{F\in\F\mid \dim(\spa F)=j\}$ denote the set of (the relative interiors of) the \mbox{$j$-dimensional} faces of~$C$, $j=0,1,\ldots,d$.

For a spherical Borel set $M^s\in\mB(S^{d-1})$ we can write the $(d-1)$-dimensional normalized Hausdorff volume, which we shall denote by $\rvol(M^s)$, in the form
\begin{equation}\label{eq:def-rvol}
  \rvol(M^s) = \underset{x\in\N(0,I_d)}{\Prob}\big[x\in M\big] \; ,
\end{equation}
where $\N(0,I_d)$ denotes the standard normal distribution on~$\IR^d$, and $M:=\{\lambda x\mid \lambda>0, x\in M^s\}\in\hmB(\IR^d)$. Clearly, the distribution $\N(0,I_d)$ may be replaced by any other orthogonal invariant distribution~$\mu$ on~$\IR^d$, which satisfies~$\mu(\{0\})=0$.

Denoting by $\Pi_C\colon\IR^d\to C$, $x\mapsto \argmin\{\|x-y\|\mid y\in C\}$ the canonical projection on~$C$, the intrinsic volumes of~$C$ are given by
\begin{equation}\label{eq:V_j(C)-polyhdrl}
  V_j(C) = \sum_{F\in\F_j} \; \underset{x\in\N(0,I_d)}{\Prob}\big[\Pi_C(x)\in F\big] \;,\quad j=0,1,\ldots,d \; .
\end{equation}
Note that $V_d(C)=\rvol(C\cap S^{d-1})$ and $V_0(C)=\rvol(\breve{C}\cap S^{d-1})$.

For $M\in\hat{\mB}(\IR^d)$ the curvature measures of $C$ evaluated in $M$ are given by
\begin{equation}\label{eq:Phi_j(C,M)-polyhdrl}
   \Phi_j(C,M) = \sum_{F\in\F_j} \; \underset{x\in\N(0,I_d)}{\Prob}\big[\Pi_C(x)\in F\cap M\big] \;,\quad j=0,1,\ldots,d \; .
\end{equation}
Note that for $j\in\{0,d\}$ we have
\begin{equation}\label{eq:Phi_d-Phi_0}
  \Phi_d(C,M) = \rvol(C\cap M\cap S^{d-1}) \;,\quad \Phi_0(C,M) = \begin{cases} V_0(C) = \rvol(\breve{C}\cap S^{d-1}) & \text{if } 0\in M \; , \\ 0 & \text{if } 0\not\in M \; . \end{cases}
\end{equation}
Additionally, we define $V_j(C):=0$ and $\Phi_j(C,M):=0$ for $j>d$.

One could use the formulas~\eqref{eq:V_j(C)-polyhdrl} and~\eqref{eq:Phi_j(C,M)-polyhdrl} to define the intrinsic volumes and curvature measures for general closed convex cones, using an approximation procedure. But a more useful definition is via a spherical version of Steiner's formula for the volume of the tube around a convex set. We will describe this in Section~\ref{sec:spher-intr-vol}.

The following well-known facts about the intrinsic volumes and the curvature measures may be verified easily for polyhedral cones using the above characterizations of $V_j$ and $\Phi_j$ in~\eqref{eq:V_j(C)-polyhdrl} and~\eqref{eq:Phi_j(C,M)-polyhdrl}.

\begin{proposition}\label{prop:facts-intrvol}
\begin{enumerate}
  \item Interpreting $C\subseteq\IR^d$ as a cone in $\IR^{d'}$ with $d'\geq d$ does not change the intrinsic volumes nor the curvature measures.
  \item \label{enum:ivol-sum=1}
        The intrinsic volumes and the curvature measures are nonnegative, and $\sum_{j=0}^d V_j(C) = 1$ if $C\subseteq\IR^d$. We have $V_j(\IR^i)=\delta_{ij}$ the Kronecker delta.
  \item The curvature measure $\Phi_j(C,.)$ is concentrated on~$C$, that is, $\Phi_j(C,M)=\Phi_j(C,M\cap C)$. Furthermore, we have $\Phi_j(C,C)=V_j(C)$.
  \item The intrinsic volumes and the curvature measures are invariant under orthogonal transformations, i.e., for $Q\in O(d)$ we have $V_j(QC) = V_j(C)$ and $\Phi_j(QC,QM) = \Phi_j(C,M)$.
  \item For the intrinsic volumes of the dual cone we have $V_j(C) = V_{d-j}(\breve{C})$.
  \item \label{enum:prod-rule}
        If $C_1,C_2$ are closed convex cones, then $V_j(C_1\times C_2) = \sum_{i=0}^j V_i(C_1)\cdot V_{j-i}(C_2)$.
        In other words, the intrinsic volumes of a product arise as the convolution of the intrinsic volumes of the components.
  \item \label{enum:Phi_j(Pi_W(C),Pi_W(M))=...}
        If $W\subseteq\IR^d$ is a linear subspace of codimension~$m$ and $\Pi_W\colon\IR^d\to W$ the orthogonal projection onto~$W$, then $\Phi_j(\Pi_W(C),\Pi_W(M))=\Phi_{j+m}(C+W^\bot,M+W^\bot)$ for $M\in\hmB(\IR^d)$.
  \item \label{enum:prob-proj}
        The probability that the projection of a Gaussian vector lies in $M\in\hat{\mB}(\IR^d)$ is given by $\underset{x\in\N(0,I_d)}{\Prob}[\Pi_C(x)\in M] = \sum_{j=0}^d \Phi_j(C,M)$. \qed
\end{enumerate}
\end{proposition}

Note that the self-duality of $\Cbn$ implies
\begin{equation}\label{eq:V_j(Cbn)=V_(d-j)(Cbn)}
  V_j(\Cbn)=V_{d_{\be,n}-j}(\Cbn) \; .
\end{equation}

An important but nontrivial property of the intrinsic volumes is the following consequence of the Gauss-Bonnet Theorem: For a closed convex cone $C\subseteq\IR^d$
\begin{equation}\label{eq:Gauss-Bonnet}
  V_1(C)+V_3(C)+V_5(C) + \ldots = \tfrac{1}{2}\cdot \chi(C\cap S^{d-1}) \; ,
\end{equation}
where $\chi$ denotes the \emph{Euler characteristic}, cf.~\cite[Sec.~4.3]{Gl} or~\cite[Thm.~6.5.5]{SW:08}. Note that $\chi(C\cap S^{d-1})=1$ if $C$ is a closed convex cone which is not a linear subspace. In this case we have
\begin{equation}\label{eq:sum-half=1/2}
  \sum_{j\text{ even}} V_j(C) = \sum_{j\text{ odd}} V_j(C) = \tfrac{1}{2} \; .
\end{equation}

\begin{remark}\label{rem:intrvol-LP}
An important example for a polyhedral cone is the positive orthant~$\IR_+^d$. Its intrinsic volumes are easily computed using the product rule~\eqref{enum:prod-rule} in Proposition~\ref{prop:facts-intrvol}: We have $V_0(\IR_+)=V_1(\IR_+)=\frac{1}{2}$, i.e., the intrinsic volumes of a $1$-dimensional ray form a symmetric Bernoulli distribution. Hence, the intrinsic volumes of the positive orthant $\IR_+^d=\IR_+\times\ldots\times\IR_+$ are the $d$-times convolution of the Bernoulli distribution, i.e., the symmetric binomial distribution $V_j(\IR_+^d)=\binom{d}{j}/2^d$.
\end{remark}

\subsection{The kinematic formula}\label{sec:kinem-form}

Euclidean versions of the kinematic formulas are well-known, cf.~for example the survey article~\cite{HS:02} and the references given therein. Spherical versions, on the other hand, are less well-known. The formulas we need in this paper may be derived from a general version due to Glasauer~\cite{Gl}, cf.~also~\cite{Gl:Summ} or~\cite[\S2.4]{HS:02} for (short) summaries. As the literature for spherical intrinsic volumes is sparse and known results sometimes hard to find, we will describe in the appendix how the kinematic formulas, as we state them here, are derived from Glasauer's result.

The uniform probability distribution on the set of $k$-dimensional subspaces of~$\IR^d$, the \emph{Grassmann manifold} $\Gr_{d,k}=\{W\subseteq\IR^d\mid W \text{ linear subspace of dimension } k\}$, is characterized as the unique probability distribution, which is invariant under the action of the orthogonal group~$O(d)$. Loosely speaking, a $k$-dimensional subspace is drawn uniformly at random, if every subspace `has the same probability'. This distribution is for example obtained as the push-forward of Gaussian matrices (of appropriate format) via taking the kernel or the image.

\begin{thm}[Kinematic formula]\label{thm:kinem-form}
Let $C\subseteq\IR^d$ be a closed convex cone, and let $W\subseteq\IR^d$ be a uniformly random subspace of codimension~$m\in\{1,\ldots,d-1\}$, i.e., $\dim W=d-m$. Furthermore, let $M\in\hat{\mB}(\IR^d)$ be such that $M\subseteq C$, and let $\Pi_W$ denote the orthogonal projection on~$W$. Then we have for the \emph{random intersection} $C\cap W$
\begin{align}
  \IE\big[\Phi_j(C\cap W,M\cap W)\big] & = \Phi_{m+j}(C,M) \;,\quad \text{for } j=1,2,\ldots,d-m \; ,
\label{eq:random-inters}
\\[1mm]\IE\big[V_0(C\cap W)\big] & = V_0(C)+V_1(C)+\ldots+V_m(C) \; ,
\label{eq:random-inters-2}
\end{align}
and for the \emph{random projection} $\Pi_W(C)$ we have
\begin{align}
  \IE\big[\Phi_j(\Pi_W(C),\Pi_W(M))\big] & = \Phi_j(C,M) \;,\quad \text{for } j=0,1,\ldots,d-m-1 \; ,
\label{eq:random-proj}
\\[1mm]\IE\big[V_{d-m}(\Pi_W(C))\big] & = V_{d-m}(C)+V_{d-m+1}(C)+\ldots+V_d(C) \; .
\label{eq:random-proj-2}
\end{align}
\end{thm}


Using the expression of the Euler characteristic given in~\eqref{eq:Gauss-Bonnet}, one obtains from~\eqref{eq:random-inters} a corollary about the probability that the random intersection $C\cap W$ is the zero set~$\{0\}$.

\begin{corollary}\label{cor:prob-CcapW=0}
Let $C\subset\IR^d$ be a closed convex cone, which is not a linear subspace. Then for $W\subseteq\IR^d$ a uniformly random subspace of codimension~$m$
\begin{equation}\label{eq:Prob[CcapW_nonempty]}
  \Prob\big[C\cap W=\{0\}\big] = 2\cdot \big(V_{m-1}(C)+V_{m-3}(C)+V_{m-5}(C)+\ldots\big) \; .
\end{equation}
\end{corollary}

\begin{proof}
The Euler characteristic $\chi(C\cap W\cap S^{d-1})$ vanishes if $C\cap W=\{0\}$ and equals~$1$ otherwise, provided $C\cap W$ is not a linear subspace. Moreover, the intersection $C\cap W$ is almost surely not a linear subspace. Therefore, we may write the probability for the event $C\cap W\neq\{0\}$ as an expectation and apply the kinematic formula
  \[ \Prob\big[C\cap W\neq\{0\}\big] = \IE\big[\chi(C\cap W\cap S^{d-1})\big] \stackrel{\eqref{eq:Gauss-Bonnet}}{=} 2\cdot \sum_{j \text{ odd}} \IE\big[V_j(C\cap W)\big] \stackrel{\eqref{eq:random-inters}}{=} 2\cdot \sum_{j \text{ odd}} V_{m+j}(C) \; . \]
As the intrinsic volumes with even/odd indices add up to~$\frac{1}{2}$ by~\eqref{eq:sum-half=1/2}, we obtain
  \[ \Prob\big[C\cap W=\{0\}\big] = 1-2\cdot \sum_{j \text{ odd}} V_{m+j}(C) \stackrel{\eqref{eq:sum-half=1/2}}{=} 2\cdot \sum_{j \text{ odd}} V_{m-j}(C) \; . \qedhere \]
\end{proof}

\subsection{Statistical properties of (hCP)}\label{sec:hCP}

We introduce the following notation. A subset $K\subseteq S^{d-1}$ of the $(d-1)$th unit sphere is called \emph{spherically convex} iff it is of the form $K=C\cap S^{d-1}$, where $C\subseteq\IR^d$ is a closed convex cone. We define the dual set of~$K$ via $\breve{K}:=\breve{C}\cap S^{d-1}$. Furthermore, we denote the projection map onto~$K$ by
\begin{equation}\label{eq:def-Pi_K}
  \Pi_K \colon S^{d-1}\setminus\breve{K}\to K \;,\quad \Pi_K(p) := \|\Pi_C(p)\|^{-1}\cdot \Pi_C(p) \; ,
\end{equation}
where $\Pi_C$ denotes the canonical projection onto~$C$. The angle $d(p,q)=\arccos \langle p,q\rangle$ between $p,q\in S^{d-1}$ defines a metric on~$S^{d-1}$. We use the notation  $d(p,K):=\min\{d(p,q)\mid q\in K\}$. Note that the dual set $\breve{K}$ may be characterized as $\breve{K}=\{p\in S^{d-1}\mid d(p,K)\geq \frac{\pi}{2}\}$.

To simplify the notation we adopt the following convention: If we maximize a function~$f$ over a set~$M$, then we denote $\Argmax\{f(x)\mid x\in M\} := \{x\in M\mid f(x)=\sup\{f(y)\mid y\in M\}\}$. If the set $\Argmax\{f(x)\mid x\in M\}$ consists of a single element only, then we denote this element by $\argmax\{f(x)\mid x\in M\}$. Similarly for $\Argmin$ and $\argmin$. Note that for $v\in S^{d-1}\setminus\breve{K}$ we have
\begin{equation}\label{eq:proj-K-argmin}
  \Pi_K(v) = \argmin\{d(p,v)\mid p\in K\} = \argmax\{ \langle p,v\rangle \mid p\in K\} \; .
\end{equation}

We will see that the homogeneous case~\eqref{eq:hCP} is easily reformulated in such a way that the kinematic formula yields the proof of Theorem~\ref{thm:hCP}. The key observation is made in the next lemma, cf.~Figure~\ref{fig:argmax-proj}.

\begin{lemma}\label{lem:argmax-proj}
Let $C\subseteq\IR^d$ be a closed convex cone, $K:=C\cap S^{d-1}$, and let $B\subset\IR^d$ denote the closed unit ball. Then for $v\in S^{d-1}$
\begin{equation}\label{eq:argmax-proj}
  \argmax\{ \langle v,x\rangle \mid x\in C\cap B\} = \begin{cases} \Pi_K(v) & \text{if } v\not\in\breve{K} \\ 0 & \text{if } v\in\inter(\breve{K}) \; . \end{cases}
\end{equation}
\end{lemma}

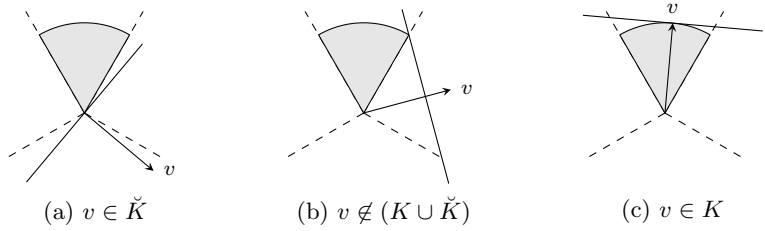
\begin{figure}
  \def\myalp{30}
  \centerline{
  \subfloat[$v\in\breve{K}$]{\begin{tikzpicture}[scale=1.2]
  \def\mybet{320};
  \def\myR{1};
  \draw[white] (-0.85,-0.8) rectangle (0.95,1.15);
  \draw[dashed] (0,0) -- (90-\myalp:1.3) (0,0) -- (90+\myalp:1.3)
                (0,0) -- (360-\myalp:1) (0,0) -- (180+\myalp:1);
  \fill[gray!20!white] (0,0) -- (90-\myalp:1) arc(90-\myalp:90+\myalp:1) -- cycle;
  \draw (0,0) -- (90-\myalp:1) arc(90-\myalp:90+\myalp:1) -- cycle;
  \draw[->,>=stealth] (0,0) -- (\mybet:\myR) node[right]{$\scriptstyle v$};
  \draw (\mybet-90:1) -- (\mybet+90:1);
  \end{tikzpicture}}
  \rule{12mm}{0mm}
  \subfloat[$v\not\in(K\cup\breve{K})$]{\begin{tikzpicture}[scale=1.2]
  \def\mybet{15};
  \def\myR{1};
  \pgfmathsetmacro\lenProj{sin(90-\myalp-\mybet)}
  \draw[white] (-0.85,-0.8) rectangle (0.95,1.15);
  \draw[dashed] (0,0) -- (90-\myalp:1.3) (0,0) -- (90+\myalp:1.3)
                (0,0) -- (360-\myalp:1) (0,0) -- (180+\myalp:1);
  \fill[gray!20!white] (0,0) -- (90-\myalp:1) arc(90-\myalp:90+\myalp:1) -- cycle;
  \draw (0,0) -- (90-\myalp:1) arc(90-\myalp:90+\myalp:1) -- cycle;
  \draw[->,>=stealth] (0,0) -- (\mybet:\myR) node[right]{$\scriptstyle v$};
  \draw (\mybet:\lenProj) ++(\mybet-90:1) -- ++(\mybet+90:2);
  \end{tikzpicture}}
  \rule{12mm}{0mm}
  \subfloat[$v\in K$]{\begin{tikzpicture}[scale=1.2]
  \def\mybet{85};
  \def\myR{1};
  \draw[white] (-0.85,-0.8) rectangle (0.95,1.15);
  \draw[dashed] (0,0) -- (90-\myalp:1.3) (0,0) -- (90+\myalp:1.3)
                (0,0) -- (360-\myalp:1) (0,0) -- (180+\myalp:1);
  \fill[gray!20!white] (0,0) -- (90-\myalp:1) arc(90-\myalp:90+\myalp:1) -- cycle;
  \draw (0,0) -- (90-\myalp:1) arc(90-\myalp:90+\myalp:1) -- cycle;
  \draw[->,>=stealth] (0,0) -- (\mybet:\myR) node[above]{$\scriptstyle v$};
  \draw (\mybet:1) ++(\mybet-90:1) -- ++(\mybet+90:2);
  \end{tikzpicture}}
  }
  \caption{An illustration of Lemma~\ref{lem:argmax-proj}.}
  \label{fig:argmax-proj}
\end{figure}

\begin{proof}
For $v\not\in\breve{K}$ there exists $p\in K$ such that $d(p,v)<\frac{\pi}{2}$, i.e., $\langle p,v\rangle>0$. It follows that
  \[ \Argmax\{ \langle v,x\rangle \mid x\in C\cap B\} = \Argmax\{ \langle v,p\rangle \mid p\in K\} \stackrel{\eqref{eq:proj-K-argmin}}{=} \{\Pi_K(v)\} \; . \]
On the other hand, if $v\in\inter(\breve{K})$, then $d(p,v)>\frac{\pi}{2}$ for all $p\in K$, i.e., $\langle p,v\rangle<0$ for all $p\in K$. It follows that in this case $\Argmax\{ \langle v,x\rangle \mid x\in C\cap B\} = \{0\}$.
%
\end{proof}

The problem~\eqref{eq:hCP} can now be phrased in the following form: We have the closed convex cone~$C$ in $d$-dimensional euclidean space~$\E$. This cone is intersected with the closed unit ball $B(\E):=\{x\in\E\mid \|x\|\leq1\}$ and with the linear subspace $W:=\{x\in\E\mid \langle a_1,x\rangle=\ldots=\langle a_m,x\rangle=0\}$. In other words, we have
  \[ \F(\hCP) = C\cap W\cap B(\E) \; . \]
If the $a_i$ are from the normal distribution $\N(\E)$ then $W$ has almost surely codimension~$m$, and~$W$ is uniformly distributed among all $(d-m)$-dimensional subspaces of $\E$. So we may assume w.l.o.g.~that $W$ is a uniformly random $(d-m)$-dimensional subspace of~$\E$.

We are now in a position to apply the kinematic formula.

\begin{proof}[Proof of Theorem~\ref{thm:hCP}]
The linear functional~$z$ to be minimized in~\eqref{eq:hCP} may be replaced by its orthogonal projection~$\bar{z}$ on $W$, as this does not change the value of the functional on~$W$. For fixed $W$ we thus obtain a conditional distribution for $\bar{z}$, which, by the well-known properties of the normal distribution, is again the normal distribution (on~$W$). Hence the probability that the origin is the solution of~\eqref{eq:hCP} is given in the following way
\begin{align*}
   \underset{a_1,\ldots,a_m,z}{\Prob} \left[\sol(\hCP)=0\right] & = \underset{W}{\Prob}\;\underset{\bar{z}}{\Prob}\big[\argmax\{ \langle\bar{z}, x\rangle \mid x\in W\cap C\cap B(\E)\}=0\big]
\\ & \stackrel{\eqref{eq:argmax-proj}}{=} \underset{W}{\Prob}\;\underset{\bar{z}}{\Prob}\big[\bar{z}\in \dual(W\cap C)\big] \stackrel{\eqref{eq:Phi_d-Phi_0}}{=} \underset{W}{\IE}\big[V_0(W\cap C)\big] \stackrel{\eqref{eq:random-inters-2}}{=} \sum_{j=0}^m V_j(C) \; ,
\end{align*}
which shows the first claim in~\eqref{eq:hCP-probs1}. As for the second claim in~\eqref{eq:hCP-probs1}, let $\Pi_{C_W}$ denote the projection onto $C_W:=C\cap W$. Then we obtain for $M\in\hmB(\E)$ such that $0\not\in M$
\begin{align*}
   \underset{a_1,\ldots,a_m,z}{\Prob} \big[\sol(\hCP)\in M\big] & = \underset{W}{\Prob}\;\underset{\bar{z}}{\Prob}\big[\argmax\{ \langle\bar{z}, x\rangle \mid x\in W\cap C\cap B(\E)\}\in M\big]
\\ & \stackrel{\eqref{eq:argmax-proj}}{=} \underset{W}{\Prob}\;\underset{\bar{z}}{\Prob}\big[\Pi_{C_W}(\bar{z})\in M\big] \; .
\end{align*}
For fixed~$W$ we obtain from Proposition~\ref{prop:facts-intrvol}\eqref{enum:prob-proj}
  \[ \underset{\bar{z}}{\Prob}\big[\Pi_{C_W}(\bar{z})\in M\big] = \sum_{j=1}^{d-m} \Phi_j(C\cap W,M) \; . \]
For random~$W$ we may apply the kinematic formula and obtain
\begin{align*}
   \underset{a_1,\ldots,a_m,z}{\Prob} \big[\sol(\hCP)\in M\big] & = \sum_{j=1}^{d-m} \underset{W}{\IE}\left[\Phi_j(C\cap W,M)\right] \stackrel{\eqref{eq:random-inters}}{=} \sum_{j=1}^{d-m} \Phi_{j+m}(C,M) \; ,
\end{align*}
which shows the second claim in~\eqref{eq:hCP-probs1}.

Finally, if the cone~$C$ is not a linear subspace, then by Corollary~\ref{cor:prob-CcapW=0}:
\begin{align*}
   \underset{a_1,\ldots,a_m}{\Prob} \; \left[\F(\hCP)=\{0\}\right] & = \underset{W}{\Prob} \; \big[C\cap W=\{0\}\big]
\\ & \stackrel{\eqref{eq:Prob[CcapW_nonempty]}}{=} 2\cdot \big(V_{m-1}(C)+V_{m-3}(C)+V_{m-5}(C)+\ldots\big) \; ,
\end{align*}
which is~\eqref{eq:hCP-probs2} and thus finishes the proof of Theorem~\ref{thm:hCP}.
\end{proof}

\subsection{Statistical properties of (CP)}\label{sec:CP}

The geometric interpretation of~\eqref{eq:CP}, which is suitable for applications of the kinematic formula, is slightly more complicated than in the homogeneous case~\eqref{eq:hCP}. The key observation is in the following lemma, which reduces the $d$-dimensional to the $2$-dimensional case.

\begin{lemma}\label{lem:red-2dim}
Let $v,w\in S^{d-1}$ be such that $\langle v,w\rangle=0$, and let $L:=\spa\{v,w\}$ denote the plane spanned by $v$ and $w$. Furthermore, let $\Pi_L\colon\IR^d\to L$ denote the orthogonal projection onto~$L$. Then for a closed convex cone $C\subseteq\IR^d$ we have
\begin{align*}
   \sup\{\langle v,x\rangle\mid x\in C \,,\; \langle w,x\rangle =1 \} & = \sup\{\langle v,x\rangle\mid x\in \Pi_L(C) \,,\; \langle w,x\rangle =1 \} \; ,
\\ \Argmax\{\langle v,x\rangle\mid x\in C \,,\; \langle w,x\rangle =1 \} & = C\cap \Pi_L^{-1}(\Argmax\{\langle v,x\rangle\mid x\in \Pi_L(C) \,,\;\langle w,x\rangle =1 \}) \; .
\end{align*}
\end{lemma}

\begin{proof}
Let $x\in\IR^d$ be decomposed in $x=x_1+x_2$ with $x_1\in L$ and $x_2\in L^\bot$, i.e., $x_1=\Pi_L(x)$. Then we have $\langle v,x\rangle = \langle v,x_1\rangle + \langle v,x_2\rangle = \langle v,x_1\rangle$, and similarly $\langle w,x\rangle = \langle w,x_1\rangle$. This implies $\sup\{\langle v,x\rangle\mid x\in C \,,\; \langle w,x\rangle =1 \} = \sup\{\langle v,x_1\rangle\mid x_1\in \Pi_L(C) \,,\; \langle w,x_1\rangle =1 \}$. Analogously, we obtain the second claim.
\end{proof}

We now discuss the $2$-dimensional case. Let $v,w\in S^1$ with $\langle v,w\rangle=0$, i.e., the matrix with columns $v,w$ lies in $O(2)$. The orthogonal group $O(2)$ is isometric to the disjoint union $S^1\,\dot{\cup}\, S^1$ according to the two possible orientations \wv and \mbox{\vw.} To make this explicit we define the map
\begin{equation}\label{eq:isom-O(2)-S^1x(l,r)}
  \vp\colon O(2)\to S^1\times\{\ell,r\} \;,\qquad Q=(v,w)\mapsto \begin{cases} (v,\ell) & \text{if the orientation is \wv} \; , \\ (v,r) & \text{if the orientation is \vw} \; . \end{cases}
\end{equation}

In the following let $\bar{C}\subset\IR^2$ be a fixed closed convex cone, which is not a linear subspace, i.e., $\bar{C}$ is a wedge with an angle between $0$ and $\pi$. Figure~\ref{fig:aff-2dim} illustrates where the intersection of~$\bar{C}$ with the affine line $\{x\in\IR^2\mid \langle x,w\rangle=1\}$ achieves its maximum with respect to the linear functional defined by~$v$, where $(v,w)\in O(2)$ according to the orientation~\wv.

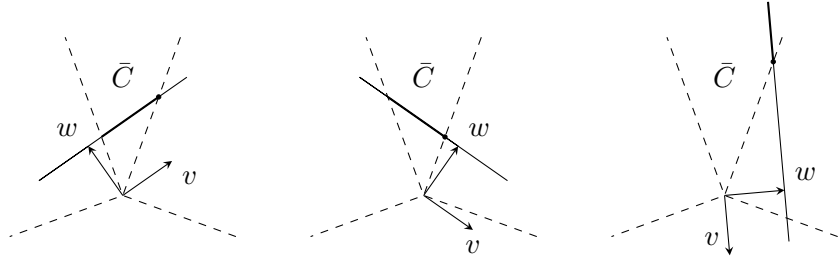
\begin{figure}
  \begin{center}
  \def\myalp{20}
  \begin{tikzpicture}[>=stealth, scale=0.8]
  \begin{scope}[xshift=-5cm]
  \def\mybet{35}
  \pgfmathsetmacro\lenA{1/cos(\mybet-\myalp)}
  \pgfmathsetmacro\lenB{1/cos(\mybet+\myalp)}
  \draw[dashed] (90-\myalp:2.8) -- (0,0) -- (90+\myalp:2.8)
        (360-\myalp:2) -- (0,0) -- (180+\myalp:2);
  \draw[->] (0,0) -- (\mybet:1) node[below right]{$v$};
  \draw[->] (0,0) -- (90+\mybet:1) node[above left]{$w$};
  \draw[thin] (90+\mybet:1) -- ++(\mybet:-1) -- ++(\mybet:3);
  \draw[thick] (90+\myalp:\lenA) -- (90-\myalp:\lenB);
  \filldraw (90-\myalp:\lenB) circle (1pt);
  \path (0,2) node{$\bar{C}$};
  \end{scope}
  \def\mybet{-35}
  \pgfmathsetmacro\lenA{1/cos(\mybet-\myalp)}
  \pgfmathsetmacro\lenB{1/cos(\mybet+\myalp)}
  \draw[dashed] (90-\myalp:2.8) -- (0,0) -- (90+\myalp:2.8)
        (360-\myalp:2) -- (0,0) -- (180+\myalp:2);
  \draw[->] (0,0) -- (\mybet:1) node[below]{$v$};
  \draw[->] (0,0) -- (90+\mybet:1) node[above right]{$w$};
  \draw[thin] (90+\mybet:1) -- ++(\mybet:-2) -- ++(\mybet:3);
  \draw[thick] (90+\myalp:\lenA) -- (90-\myalp:\lenB);
  \filldraw (90-\myalp:\lenB) circle (1pt);
  \path (0,2) node{$\bar{C}$};
  \begin{scope}[xshift=5cm]
  \def\mybet{-85}
  \pgfmathsetmacro\lenA{1/cos(\mybet-\myalp)}
  \pgfmathsetmacro\lenB{1/cos(\mybet+\myalp)}
  \draw[dashed] (90-\myalp:2.8) -- (0,0) -- (90+\myalp:2.8)
        (360-\myalp:2) -- (0,0) -- (180+\myalp:2);
  \draw[->] (0,0) -- (\mybet:1) node[above left]{$v$};
  \draw[->] (0,0) -- (90+\mybet:1) node[above right]{$w$};
  \draw[thin] (90-\myalp:\lenB) -- ++(\mybet:3);
  \draw[thick] (90-\myalp:\lenB) -- ++(\mybet:-1);
  \filldraw (90-\myalp:\lenB) circle (1pt);
  \path (0,2) node{$\bar{C}$};
  \end{scope}
  \end{tikzpicture}
  \end{center}
  \caption{The $2$-dimensional situation in the inhomogeneous case.}
  \label{fig:aff-2dim}
\end{figure}

Let the two rays forming the boundary of $\bar{C}$ be denoted by $R_1$ and $R_2$. Furthermore, depending on $v,w$, write 
\begin{equation}\label{eq:not-F,val,sol}
  \bar{\F} :=\{x\in \bar{C} \mid \langle w,x\rangle=1\} ,\;\; \ol{\val} := \sup\{\langle v,x\rangle \mid x\in \bar{\F}\} ,\;\; \ol{\Sol} := \Argmax\{\langle v,x\rangle \mid x\in \bar{\F}\} .
\end{equation}
Assuming that $v,w$ are random vectors, such that $(v,w)\in O(2)$ uniformly at random, it is easily seen that only four cases appear with positive probability: The intersection $\bar{\F}$ may be empty, the functional~$v$ may be unbounded on~$\bar{\F}$, or the solution set $\ol{\Sol}$ consists of a single point, which either lies in $R_1$ or in~$R_2$. In the latter case we again adopt the convention to denote the single point by $\ol{\sol}$. This case distinction corresponds to the decomposition of $O(2)\cong S^1 \times \{\ell,r\}$ indicated in Figure~\ref{fig:sol-aff-2dim}, where we use the isometry~$\vp$ defined in~\eqref{eq:isom-O(2)-S^1x(l,r)}.
%
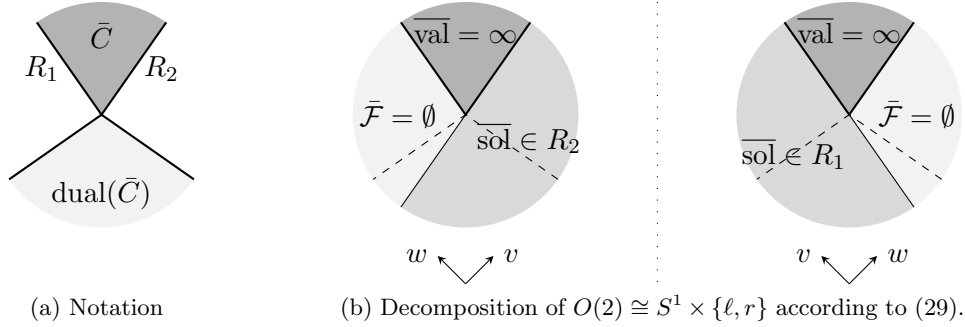
\begin{figure}
  \begin{center}
  \def\myscal{1.5}
  \def\myalp{35}
  \subfloat[Notation]{
  \begin{tikzpicture}[scale=\myscal,>=stealth]
  \draw[white] (0,1,) -- (0,-1.5);
  \fill[gray!60!white] (90-\myalp:1) -- (0,0) -- (90+\myalp:1) arc(90+\myalp:90-\myalp:1);
  \fill[gray!10!white] (360-\myalp:1) arc(360-\myalp:180+\myalp:1) -- (0,0) -- cycle;
  \draw[thick] (90-\myalp:1) -- (0,0) -- (90+\myalp:1)
        (360-\myalp:1) -- (0,0) -- (180+\myalp:1);
  \path (0,0.7) node{$\bar{C}$}
        (0,-0.7) node{$\dual(\bar{C})$}
        (90+\myalp:0.5) node[left]{$R_1$}
        (90-\myalp:0.5) node[right]{$R_2$};
  \end{tikzpicture}}
  \rule{18mm}{0mm}
  \subfloat[Decomposition of $O(2)\cong S^1\times\{\ell,r\}$ according to~\eqref{eq:not-F,val,sol}.]{
  \begin{tikzpicture}[scale=\myscal,>=stealth]
  \begin{scope}[xshift=-1.7cm]
  \begin{scope}[shift={(0,-1.5)}, scale=0.35]
    \draw[->] (0,0) -- (45:1) node[right]{$v$};
    \draw[->] (0,0) -- (135:1) node[left]{$w$};
  \end{scope}
  \fill[gray!60!white] (90-\myalp:1) -- (0,0) -- (90+\myalp:1) arc(90+\myalp:90-\myalp:1);
  \fill[gray!30!white] (90-\myalp:1) arc(90-\myalp:0:1) arc(360:270-\myalp:1) -- cycle;
  \fill[gray!10!white] (270-\myalp:1) arc(270-\myalp:90+\myalp:1) -- (0,0) -- cycle;
  \draw[thick] (90-\myalp:1) -- (0,0) -- (90+\myalp:1);
  \draw[dashed] (360-\myalp:1) -- (0,0) -- (180+\myalp:1);
  \draw[thin] (0,0) -- (270-\myalp:1);
  \path 
        (0,0.75) node{$\ol{\val}=\infty$}
        (360-\myalp+15:0.6) node{$\ol{\sol}\in R_2$}
        (180:0.6) node{$\bar{\F}=\emptyset$};
  \end{scope}
  \draw[loosely dotted] (0,1) -- (0,-1.5);
  \begin{scope}[xshift=1.7cm]
  \begin{scope}[shift={(0,-1.5)}, scale=0.35]
    \draw[->] (0,0) -- (45:1) node[right]{$w$};
    \draw[->] (0,0) -- (135:1) node[left]{$v$};
  \end{scope}
  \fill[gray!60!white] (90-\myalp:1) -- (0,0) -- (90+\myalp:1) arc(90+\myalp:90-\myalp:1);
  \fill[gray!30!white] (90+\myalp:1) arc(90+\myalp:270+\myalp:1) -- cycle;
  \fill[gray!10!white] (270+\myalp:1) arc(270+\myalp:360:1) arc(0:90-\myalp:1) -- (0,0) -- cycle;
  \draw[thick] (90-\myalp:1) -- (0,0) -- (90+\myalp:1);
  \draw[dashed] (360-\myalp:1) -- (0,0) -- (180+\myalp:1);
  \draw[thin] (0,0) -- (270+\myalp:1);
  \path 
        (0,0.75) node{$\ol{\val}=\infty$}
        (0:0.6) node{$\bar{\F}=\emptyset$}
        (215:0.6) node{$\ol{\sol}\in R_1$};
  \end{scope}
  \end{tikzpicture}}
  \end{center}
  \caption{Illustration of Lemma~\ref{lem:aff-d=2}.}
  \label{fig:sol-aff-2dim}
\end{figure}

We formulate the following lemma, which is checked easily, cf.~Figure~\ref{fig:sol-aff-2dim}.

\begin{lemma}\label{lem:aff-d=2}
Let $\bar{C}\subset\IR^2$ be a closed convex cone, which is not a linear subspace, and let $R_1$ and $R_2$ denote the two rays forming the boundary of~$\bar{C}$. Then, for uniformly random $(v,w)\in O(2)$, we have
\begin{equation}\label{eq:lem-2dim-empty,unbound}
  \Prob\big[\bar{\F}=\emptyset\big] = V_0(\bar{C}) \;,\qquad \Prob\Big[\ol{\val}=\infty\Big] = V_2(\bar{C}) \; ,
\end{equation}
where we use the notation from~\eqref{eq:not-F,val,sol}. Furthermore, for $M\in\hmB(\IR^2)$, we have
\begin{equation}\label{eq:lem-2dim-rank-pos,rank}
  \Prob\Big[\ol{\sol}\in M \Big] = \Phi_1(\bar{C},M) \;,\qquad \Prob\Big[\ol{\sol}\in M \text{ and } \ol{\val}>0\Big] = \tfrac{1}{2}\cdot \Phi_1(\bar{C},M) \; . \qed
\end{equation}
\end{lemma}

As in the homogeneous case, we will now bring the problem~\eqref{eq:CP} into a geometric form where we can apply the kinematic formula. This requires a few more steps than for~\eqref{eq:hCP}.

In addition to $W:=\{x\in\E\mid \langle a_1, x\rangle=\ldots=\langle a_m, x\rangle=0\}$ we introduce the following notation
  \[ W_\aff := \{x\in\E\mid \langle a_1, x\rangle=b_1,\ldots,\langle a_m, x\rangle=b_m\} \;,\qquad \wh{W} := \spa(W_\aff) \; . \]
For normal distributed $a_1,\ldots,a_m,b_1,\ldots,b_m$, the set $W_\aff$ is almost surely an affine space of codimension~$m$ and its linear hull $\wh{W}$ is a linear space of codimension~$m-1$. Moreover, the probability distribution of the $a_i$ and $b_i$ induce for $\wh{W}$ a uniform distribution on the Grassmann manifold of subspaces of $\E$ with codimension~$m-1$.

The affine space $W_\aff$ has almost surely a positive \emph{height} $\min\{\|x\|\mid x\in W_\aff\}$. We define the normalization $W_\aff^\circ$ of $W_\aff$, which has unit height, via
  \[ W_\aff^\circ=h^{-1}\cdot W_\aff \;,\qquad h=\min\{\|x\|\mid x\in W_\aff\} \; . \]
Additionally, we define the point $w\in W_\aff^\circ$ by the property $W_\aff^\circ\cap S(\wh{W})=\{w\}$.
So we have $W=\wh{W}\cap w^\bot$ and $W_\aff^\circ=w+W$. See Figure~\ref{fig:ilstr_aff} for an illustration of these definitions.
\begin{figure}
  \def\myalp{140}
  \centerline{\begin{tikzpicture}[scale=1]
  \filldraw (0,0) circle(1pt) node[below right]{$0$};
  \draw (0,0) circle(1cm);
  \filldraw (\myalp:1) circle(1.5pt) node[below right]{$w$};
  \draw (\myalp:1.8) -- ++(\myalp+90:1.5) -- ++(\myalp-90:3) node[left]{$W_\aff$};
  \draw (\myalp:1) -- ++(\myalp+90:1.5) -- ++(\myalp-90:3) node[right]{$W_\aff^\circ$};
  \draw (0,0) -- ++(\myalp+90:1.5) -- ++(\myalp-90:3) node[below right]{$W$};
  \path (2.5,1) node{$\wh{W}$}
        (0.9,-0.2) node[right]{$S(\wh{W})$};
  \end{tikzpicture}}
  \caption{An illustration of the geometric situation.}
  \label{fig:ilstr_aff}
\end{figure}
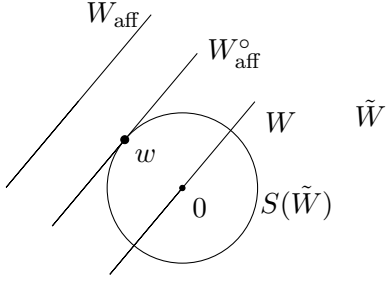

As we are not interested in the specific value of~\eqref{eq:CP} (provided it is $<\infty$) but only where the maximum is attained, we may consider $W_\aff^\circ$ instead of $W_\aff$, i.e., instead of~\eqref{eq:CP} we consider
\begin{equation}\label{eq:alt-SDP-1}
  \text{maximize} \quad z\cdot x \qquad \text{s.t.} \quad x\in C\cap\wh{W} \,,\; \langle w, x\rangle =1 \; .
\end{equation}
It is easily seen that for fixed $\wh{W}$ the induced distribution of $w$ is the uniform distribution on~$S(\wh{W})$.

Without loss of generality, we may replace the functional~$z$ by its orthogonal projection~$\bar{z}$ on~$W$. For fixed~$W$ the induced distribution of~$\bar{z}$ is the normal distribution on~$W$. As~$\bar{z}$ is almost surely nonzero, we may define the normalization $v:=\|\bar{z}\|^{-1}\cdot \bar{z}\in S(W)$. Finally, we denote the plane spanned by $v,w$ by $L:=\spa\{v,w\}$.

We can generate the distribution of $(\tilde{W},L,v,w)$ induced by the standard normal distributed $a_1,\ldots,a_m,b_1,\ldots,b_m$ in the following way:
\begin{enumerate}
  \item choose a uniformly random subspace $\wh{W}$ of $\E$ of codimension $m-1$,
  \item choose a plane $L\subseteq\tilde{W}$ uniformly at random,
  \item choose $v\in S(L)$ uniformly at random,
  \item choose $w$ as one of the points in $S(L)\cap v^\bot$, each with probability~$\frac{1}{2}$.
\end{enumerate}

We may now proceed to the proof of Theorem~\ref{thm:CP}.

\begin{proof}[Proof of Theorem~\ref{thm:CP}]
Lemma~\ref{lem:red-2dim} tells us that instead of~\eqref{eq:alt-SDP-1} we may consider the following problem in the $2$-dimensional plane~$L$
\begin{equation}\label{eq:alt-SDP-2}
  \text{maximize} \quad \langle v, x\rangle \qquad \text{s.t.} \quad x\in\Pi_L(C\cap\wh{W}) \,,\; \langle w, x\rangle =1 \; .
\end{equation}
More precisely, using the notation of~\eqref{eq:not-F,val,sol}
\begin{align*}
   \bar{C} & := \Pi_L(C\cap\wh{W}) \;, & \bar{\F} & :=\{x\in \bar{C} \mid \langle w, x\rangle=1\} \; ,
\\ \ol{\val} & := \sup\{ \langle v, x\rangle \mid x\in \bar{\F}\} \;, & \ol{\Sol} & := \Argmax\{ \langle v, x\rangle \mid x\in \bar{\F}\} ,
\end{align*}
we obtain from Lemma~\ref{lem:red-2dim} that \eqref{eq:CP} is infeasible iff $\bar{\F}=\emptyset$, \eqref{eq:CP} is unbounded iff $\ol{\val}=\infty$, and
\begin{equation}\label{eq:CP--2-dim}
  \Sol(\CP) = C\cap \Pi_L^{-1}(\ol{\Sol}) \; . 
\end{equation}
We thus obtain by Lemma~\ref{lem:aff-d=2}
  \[ \underset{\substack{a_1,\ldots,a_m \\ b_1,\ldots,b_m}}{\Prob} [\CP \text{ infeasible}] = \underset{\wh{W},L}{\Prob}\;\underset{v,w}{\Prob} \big[\bar{\F}=\emptyset\big] \stackrel{\eqref{eq:lem-2dim-empty,unbound}}{=} \underset{\wh{W},L}{\IE}\big[V_0(\Pi_L(C\cap\wh{W}))\big] \; . \]
Applying the kinematic formula twice yields (recall that the codimension of $\tilde{W}$ is $m-1$)
  \[ \underset{\wh{W},L}{\IE}\big[V_0(\Pi_L(C\cap\wh{W}))\big] \stackrel{\eqref{eq:random-proj}}{=} \underset{\wh{W}}{\IE}\big[V_0(C\cap\wh{W})\big] \stackrel{\eqref{eq:random-inters-2}}{=} V_0(C)+V_1(C)+\ldots+V_{m-1}(C) \; , \]
which proves the first claim in~\eqref{eq:CP-probs1}. Analogously, we obtain
\begin{align*}
   \underset{\substack{a_1,\ldots,a_m,z\\b_1,\ldots,b_m}}{\Prob} [\CP & \text{ unbounded}] = \underset{\wh{W},L}{\Prob}\;\underset{v,w}{\Prob} \big[\ol{\val}=\infty\big] \stackrel{\eqref{eq:lem-2dim-empty,unbound}}{=} \underset{\wh{W},L}{\IE}\big[V_2(\Pi_L(C\cap\wh{W}))\big]
\\ & \stackrel{\eqref{eq:random-proj-2}}{=} \underset{\wh{W}}{\IE}\big[V_2(C\cap\wh{W})+V_3(C\cap\wh{W})+\ldots+V_{d-m+1}(C\cap\wh{W})\big]
\\ & \stackrel{\eqref{eq:random-inters}}{=} V_{m+1}(C)+V_{m+2}(C)+\ldots+V_d(C) \; ,
\end{align*}
which proves the second claim in~\eqref{eq:CP-probs1}.

As for the claim~\eqref{eq:CP-probs2}, we have for $M\in\hmB(\E)$ (cf.~\eqref{eq:CP--2-dim} above)
  \[ \big(\sol(\CP)\in M\cap\wh{W} \text{ and } \val(\CP)>0\big) \iff \big(\ol{\sol}\in \Pi_L(M\cap\wh{W}) \text{ and } \ol{\val}>0 \big) \; . \]
This yields
\begin{align*}
   \underset{\substack{a_1,\ldots,a_m,z\\b_1,\ldots,b_m}}{\Prob} [\sol(\CP)\in M & \text{ and } \val(\CP)>0] = \underset{\wh{W},L}{\Prob}\;\underset{v,w}{\Prob} \Big[\ol{\sol}\in \Pi_L(M\cap\wh{W}) \text{ and } \ol{\val}>0\Big]
\\ & \stackrel{\eqref{eq:lem-2dim-rank-pos,rank}}{=} \underset{\wh{W},L}{\IE}\big[\tfrac{1}{2}\cdot \Phi_1(\Pi_L(C\cap\wh{W}),\Pi_L(M\cap\wh{W}))\big] \; .
\end{align*}
Applying the kinematic formula twice finally yields (recall $\codim \tilde{W}=m-1$)
\begin{align*}
   \underset{\wh{W},L}{\IE}\big[\tfrac{1}{2}\cdot \Phi_1(\Pi_L(C\cap\wh{W}),\Pi_L(M\cap\wh{W}))\big] & \stackrel{\eqref{eq:random-proj}}{=} \tfrac{1}{2}\cdot \underset{\wh{W}}{\IE}\big[\Phi_1(C\cap\wh{W},M\cap\wh{W})\big] \stackrel{\eqref{eq:random-inters}}{=} \tfrac{1}{2}\cdot \Phi_m(C,M) \; .
\end{align*}
Analogous arguments yield the claim without the positivity assumption.
\end{proof}

\section{Main result}\label{sec:main_res}

In this section we will formulate our main result, which are closed formulas for the curvature measures of the cone of positive semidefinite matrices over~$\IR/\IC/\IH$ evaluated at the set of rank~$r$ matrices. We first give the definition of intrinsic volumes and curvature measures of arbitrary closed convex cones in Section~\ref{sec:spher-intr-vol}. Section~\ref{sec:Mehta-and-related} is devoted to Mehta's and related integrals. The notation we establish here will allow us to formulate our main result in Section~\ref{sec:form-Phi_j(be,n,r)}. In Section~\ref{sec:ex-small-dims} we will give some examples in small dimensions.

\subsection{Spherical intrinsic volumes}\label{sec:spher-intr-vol}

Central to the definition of the intrinsic volumes and the curvature measures is the notion of the (local) tube around a spherically convex set $K\subseteq S^{d-1}$. For a spherical Borel set $M^s\in\mB(S^{d-1})$ we define the (local) tube around $K$ (in $M^s$) of radius $\alpha\in[0,\pi/2)$ via
\begin{align*}
   \T(K,\alpha) & := \{ p\in S^{d-1}\mid d(p,K)\leq \alpha\} \; ,
\\ \T(K,\alpha;M^s) & := \{ p\in \T(K,\alpha)\mid \Pi_K(p)\in M^s\} \; ,
\end{align*}
where $\Pi_K$ denotes the spherical projection map, cf.~\eqref{eq:def-Pi_K}. We oppress the dependence on $S^{d-1}$ to keep the notation simple. Recall that we denote the $(d-1)$-dimensional normalized Hausdorff volume on $S^{d-1}$ by~$\rvol$, cf.~\eqref{eq:def-rvol}.
%

The following proposition forms the basis for the general definition of the curvature measures and the intrinsic volumes. For a proof see for example~\cite{H:43,alle:48,S:50,K:91,Gl}.

\begin{proposition}\label{prop:rvol(T(K,a;M))}
Let $K\subseteq S^{d-1}$ be a spherically convex set. Then there exist nonnegative measures $\Phi_0(K,.),\Phi_1(K,.),\ldots,\Phi_{d-1}(K,.)\colon\mB(S^{d-1})\to\IR_+$ such that for $0\leq\alpha<\pi/2$ and $M^s\in\mB(S^{d-1})$
\begin{equation}\label{eq:rvol(T(K,a;M))}
  \rvol \T(K,\alpha;M^s) = \Phi_{d-1}(K,M^s) + \sum_{j=1}^{d-1} \Phi_{j-1}(K,M^s)\cdot \rvol \T(S^{j-1},\alpha) \; .
\end{equation}
Furthermore, if $K_1,K_2,\ldots\subseteq S^{d-1}$ is a sequence of spherically convex sets, which converges to~$K$ in the Hausdorff metric, then $\lim_{\ell\to\infty} \Phi_{j-1}(K_\ell,M^s)\to \Phi_{j-1}(K,M^s)$ for all $M^s\in\mB(S^{d-1})$.
\end{proposition}

\begin{definition}\label{def:Phi_j(C,M),V_j(C)}
Let $C\subseteq\IR^d$ be a closed convex cone, and let $K=C\cap S^{d-1}$ be the corresponding spherically convex set. For $1\leq j\leq d$, the function $\Phi_{j-1}(K,.)$ from Proposition~\ref{prop:rvol(T(K,a;M))} is called the $(j-1)$th \emph{curvature measure of $K$}. The \emph{intrinsic volumes of $K$} are defined by
  \[ V_{j-1}(K) := \Phi_{j-1}(K,S^{d-1}) \; . \]

Let the conic $\sigma$-algebra $\hmB(\IR^d)$ be defined as in Section~\ref{sec:intrvol-polyhedr}. For $1\leq j\leq d$ the $j$th \emph{curvature measure of~$C$} is the functional $\Phi_j(C,.)\colon\hmB(\IR^d)\to\IR_+$, which is given by
  \[ \Phi_j(C,M) := \Phi_{j-1}(C\cap S^{d-1},M\cap S^{d-1}) \; . \]
The curvature measure $\Phi_0(C,.)\colon\hmB(\IR^d)\to\IR_+$ is defined by the scaled Dirac measure
  \[ \Phi_0(C,M) := \begin{cases} \rvol(\breve{C}\cap S^{d-1}) & \text{if } 0\in M \; , \\ 0 & \text{if } 0\not\in M \; . \end{cases} \]
The \emph{intrinsic volumes of~$C$} are defined by
  \[ V_j(C) := \Phi_j(C,C) \;,\quad j=0,\ldots,d \; . \]
\end{definition}

\begin{remark}\label{rem:index-shift}
Note that we introduce a shift in the index of the curvature measures/intrinsic volumes when passing between the cone and its intersection with the unit sphere. We do this for several reasons. First of all, the spherical notation is established this way at several places~\cite{Gl,Gl:Summ,HG:02,am:thesis}. Second, if interpreted correctly, $V_j(C)$ and $V_{j-1}(K)$ may indeed be seen as volumes of sets with ``dimension'' $j$ and $j-1$, respectively. And last but not least, we have made the experience that the formulas get nicer when working with the shifted index for the cones, cp.~the kinematic formulas in Section~\ref{sec:kinem-form} and in Section~\ref{sec:kinem-form-supp-meas} in the appendix.
\end{remark}

\subsection{Mehta's and related integrals}\label{sec:Mehta-and-related}

For $z=(z_1,\ldots,z_n)$ we denote the Vandermonde determinant by $\Delta(z) := \prod_{1\leq i<j\leq n} (z_i-z_j)$. For $0\leq r\leq n$ let $x:=(z_1,\ldots,z_r)$ and $y:=(z_{r+1},\ldots,z_n)$, so that $z=(x,y)$. We have the decomposition
\begin{equation}\label{eq:decomp-Vanderm}
  \Delta(z)^\be = \Delta(x)^\be\cdot \Delta(y)^\be\cdot \prod_{i=1}^r \prod_{j=1}^{n-r} (x_i-y_j)^\be \; .
\end{equation}
We regard the rightmost factor in~\eqref{eq:decomp-Vanderm} as a polynomial in~$x$ and decompose it into its homogeneous parts. For convenience, we change the sign, and define
  \[ f_{\be,k}(x;y) := \bigg( \text{the $x$-homogeneous part of } \prod_{i=1}^r \prod_{j=1}^{n-r} (x_i+y_j)^\be \text{ of degree } k\bigg) \; . \]
We can write $f_{\be,k}(x;y)$ in an explicit form if we rearrange
  \[ \prod_{i=1}^r \prod_{j=1}^{n-r} (x_i+y_j)^\be = \prod_{i=1}^r \prod_{j=1}^{n-r} \big(\tfrac{x_i}{y_j}+1\big)^\be \cdot \prod_{j=1}^{n-r} y_j^{\be r} \; . \]
Denoting by $\sigma_k$ the $k$th elementary symmetric function, we obtain
\begin{equation}\label{eq:f_(b,k)(x;y)-expl}
  f_{\be,k}(x;y) = \sigma_k\left((x\otimes y^{-1})^{\times \be}\right)\cdot \prod_{j=1}^{n-r} y_j^{\be r} \; ,
\end{equation}
where $(x\otimes y^{-1})^{\times \be} = \big(\underbrace{x\otimes y^{-1},\ldots,x\otimes y^{-1}}_{\be\text{-times}}\big)$, and
\begin{equation}\label{eq:xotimesy}
  x\otimes y^{-1} := \left(\frac{x_1}{y_1},\ldots,\frac{x_r}{y_1},\frac{x_1}{y_2},\ldots,\frac{x_r}{y_2} \,,\ldots,\, \frac{x_1}{y_{n-r}},\ldots,\frac{x_r}{y_{n-r}}\right)\in\IR^{r(n-r)} \; .
\end{equation}
We thus have a decomposition of the $\be$th power of the Vandermonde determinant into
\begin{equation}\label{eq:decomp-Vanderm-f}
  \Delta(z)^\be = \Delta(x)^\be\cdot \Delta(y)^\be\cdot \sum_{k=0}^{\be r(n-r)} f_{\be,k}(x;-y) \; ,
\end{equation}
where $z=(x,y)$.

\begin{definition}
We define for $0\leq r\leq n$ and $0\leq k\leq \be r(n-r)$ the integrals
\begin{equation}\label{eq:def-J}
  J_\be(n,r,k) := \frac{1}{(2\pi)^{n/2}}\cdot \underset{z\in\IR^n_+}{\int} e^{-\frac{\|z\|^2}{2}} \cdot |\Delta(x)|^\be\cdot |\Delta(y)|^\be\cdot f_{\be,k}(x;y) \;dz \; ,
\end{equation}
where $z=(x,y)$ with $x\in\IR^r$, $y\in\IR^{n-r}$, and $\IR_+^n$ denotes the positive orthant in~$\IR^n$. We set $J_\be(n,r,k) := 0$, if $k<0$ or $k>\be r(n-r)$.
\end{definition}

Note that for $n=1$ we have
\begin{equation}\label{eq:J_b(1,0,0)=J_b(1,1,0)=...}
  J_\be(1,0,0) = J_\be(1,1,0) = \tfrac{1}{2} \; .
\end{equation}
Note also that exchanging the roles of~$x$ and~$y$ yields the following symmetry relation
\begin{equation}\label{eq:symm-J}
  J_\be(n,r,k) = J_\be(n,n-r,\be r(n-r)-k) \; .
\end{equation}

For $r\in\{0,n\}$ and $k=0$ we obtain the integrand $e^{-\frac{\|z\|^2}{2}} \cdot |\Delta(z)|^\be$, which also appears in \emph{Mehta's integral}
\begin{equation}\label{eq:F_n(b/2)}
  F_n(\be/2) := \frac{1}{(2\pi)^{n/2}}\cdot \underset{z\in\IR^n}{\int} e^{-\frac{\|z\|^2}{2}} \cdot |\Delta(z)|^\be \;dz  = \prod_{j=1}^n \frac{\Gamma(1+\frac{j\be}{2})}{\Gamma(1+\frac{\be}{2})} = n!\cdot \prod_{j=1}^n \frac{\Gamma(\frac{j\be}{2})}{\Gamma(\frac{\be}{2})}
\end{equation}
(cf.~\cite{FW:08} and the references therein).

It is well-known that the distribution of the joint probability density function for the eigenvalues of matrices from the $\GbE$ is given by (cf.~\cite{FW:08} and the references therein)
\begin{equation}\label{eq:distr-eval-GbE}
  \frac{1}{(2\pi)^{n/2}\cdot F_n(\be/2)}\cdot e^{-\frac{\|z\|^2}{2}} \cdot |\Delta(z)|^\be \; .
\end{equation}
Using this, we may write the $J$-integral as an expectation
\begin{equation}\label{eq:J-expect}
  J_\be(n,r,k) = F_r(\be/2)\cdot F_{n-r}(\be/2)\cdot \underset{\substack{A\in\GbE(r)\\B\in\GbE(n-r)}}{\IE}\big[1_+(A)\cdot 1_+(B)\cdot f_{\beta,k}(A;B)\big] \; ,
\end{equation}
where $1_+(A):=1$ if $A\seq0$ and $1_+(A):=0$ if $A\not\seq0$, and $f_{\beta,k}(A;B)$ is the value of $f_{\beta,k}$ in the eigenvalues of $A$ and $B$.

If we denote
  \[ F_n^+(\be/2) := J_\be(n,0,0) = \frac{1}{(2\pi)^{n/2}}\cdot \underset{z\in\IR_+^n}{\int} e^{-\frac{\|z\|^2}{2}} \cdot |\Delta(z)|^\be \;dz \; , \]
then we have $F_1^+(\be/2)=\frac{1}{2}$, and for $n=2,3$ and $\be=1,2,4$
\begin{equation}\label{eq:F_3^+(b)}
  F_n^+(\be/2) = \; \begin{array}{|c||c|c|c|}
  \hline \raisebox{-2mm}{\rule{0mm}{6mm}}\tikz[baseline=1.5mm, xscale=0.8, yscale=0.9]{
  \draw (0,0.5) node[below=3.5mm, right=0mm]{$n$} -- (1,0) node[above=3.5mm, left=0mm]{$\be$};
  } & 1 & 2 & 4
\\\hline & & & \\[-4.6mm]\hline\rule{0mm}{5mm}2 & \frac{1}{\sqrt{\pi}}\Big(1-\frac{\sqrt{2}}{2}\Big) & \frac{1}{2}-\frac{1}{\pi} & 3-\frac{8}{\pi}
\\[2mm]\hline\rule{0mm}{5mm} 3 & \frac{1}{\sqrt{\pi}}\Big(\frac{3}{4}-\frac{3\sqrt{2}}{2\pi}\Big) & \frac{3}{2}-\frac{9}{2\pi} & 540-\frac{1692}{\pi}
\\\hline
\end{array} \; .
\end{equation}
In fact, the values for $F_2^+(\be/2)$ are easily computed with any computer algebra system; we obtained the values for $F_3^+(\be/2)$ differently (cf.~Section~\ref{sec:ex-small-dims} for the details). As for the asymptotics, it is shown in~\cite{DM:06} that for $n\to\infty$
  \[ F_n^+(\be/2) = \Theta\left(\exp\left(-n^2\cdot \frac{\be \ln 3}{4}\right)\right) \; . \]

\subsection{Formulas for the curvature measures of the rank \texorpdfstring{$r$}{r} stratum}\label{sec:form-Phi_j(be,n,r)}

Recall the values $\Phi_j(\be,n,r)$ of the $j$th curvature measures of~$\Cbn$ evaluated at the set of its rank~$r$ matrices, cf.~\eqref{eq:def-Phi_j(b,n,r)}. The following theorem is a main result of this paper.

\begin{thm}\label{thm:main}
Let $\be\in\{1,2,4\}$, $n\in\IN$, and $0\leq r\leq n$. The curvature measures of the semidefinite cone $\C_{\be,n}$ evaluated at the set of rank $r$ matrices, cf.~\eqref{eq:def-C_(b,n)}--\eqref{eq:def-Phi_j(b,n,r)}, are given by
\begin{equation}\label{eq:Phi_j(C,M)}
  \Phi_j(\be,n,r) = \binom{n}{r}\cdot \frac{J_\be(n,r, j - d_{\be,r})}{F_n(\beta/2)} \; ,
\end{equation}
where $0\leq j\leq d_{\be,n}$, and $d_{\be,n}$, $J_\be(n,r,k)$, $F_n(\be/2)$ are defined in~\eqref{eq:def-d_(b,n)},~\eqref{eq:symm-J}, and~\eqref{eq:F_n(b/2)}, respectively.
\end{thm}

Note that the intrinsic volumes of the cone $\Cbn$, are given by $V_j(\C_{\be,n}) = \sum_{r=0}^n \Phi_j(\be,n,r)$, cf.~\eqref{eq:decomp-V_j(Cbn)}. Using the expression of $J_\be(n,r,k)$ from~\eqref{eq:J-expect}, we may alternatively write the curvature measures in the form
\begin{equation}\label{eq:Phi_j(C,M)-expect}
  \Phi_j(\be,n,r) = \binom{n}{r}\cdot \frac{F_r(\be/2)\cdot F_{n-r}(\be/2)}{F_n(\beta/2)} \cdot \underset{\substack{A\in\GbE(r)\\B\in\GbE(n-r)}}{\IE}\big[1_+(A)\cdot 1_+(B)\cdot f_{\beta,j-d_{\be,r}}(A;B)\big] \; .
\end{equation}

\begin{remark}\label{rem:connect-algdeg}
\begin{enumerate}
  \item The integral $J_\be(n,r,k)$ is nonzero iff $0\leq k\leq \be r(n-r)$. This implies
  \begin{equation}\label{eq:Pat-ineq_Phi}
     \Phi_j(\be,n,r)>0 \iff d_{\be,r}\leq j\leq d_{\be,r}+\be r(n-r) \; .
  \end{equation}
        One can deduce the implication ``$\Rightarrow$'' also from Corollary~\ref{cor:SDP}, which provides a different characterization of~$\Phi_j(\be,n,r)$. Namely, it is well-known (at least in the real case $\be=1$), that a random instance
 of~\eqref{eq:SDP} almost surely satisfies $d_{\be,r}\leq m\leq d_{\be,r}+\be r(n-r)$, where $r$ denotes the rank of the solution of~\eqref{eq:SDP}. These inequalities are known as \emph{Pataki's inequalities}, cf.~\cite{AHO:97,Pa:00,NRS:09}.
  \item Note that the relation $J_\be(n,r,k) = J_\be(n,n-r,\be r(n-r)-k)$, cf.~\eqref{eq:symm-J}, implies the symmetry
  \begin{equation}\label{eq:symm-Phi_j}
    \Phi_j(\be,n,r) = \Phi_{d_{\be,n}-j}(\be,n,n-r) \; .
  \end{equation}
        This is a refinement of the duality relation $V_j(\C_{\be,n})=V_{d_{\be,n}-j}(\C_{\be,n})$, which follows from the self-duality of~$\Cbn$, cf.~Section~\ref{sec:intrvol-polyhedr}.
  \item Both of the above properties of $\Phi_j(\be,n,r)$ also hold for the \emph{algebraic degree of semidefinite programming}, cf.~\cite[Prop.~9]{NRS:09}. We conjecture a deeper reason for this coincidence, which should be interesting to explore further.
\end{enumerate}
\end{remark}

%

\begin{remark}\label{rem:completeness}
With Theorem~\ref{thm:main} we have formulas for the intrinsic volumes of almost all symmetric cones: Recall that the characterization theorem of symmetric cones says that every symmetric cone is a direct product of Lorentz cones $\mcL^n:=\{x\in\IR^n\mid x_n\geq (x_1^2+\ldots+x_{n-1}^2)^{1/2}\}$, the cones $\C_{\be,n}$, and the exceptional $27$-dimensional cone of positive semidefinite $(3\times3)$-matrices over the octonions, cf.~\cite{FK:94}. The intrinsic volumes of $\mcL^n$ can be shown to be (cf.~for example~\cite[Ex.~2.15]{ambu:11b})
\begin{align*}
   V_j(\mcL^n) & = \frac{\binom{(n-2)/2}{(j-1)/2}}{2^{n/2}} \;,\quad \text{for $1\leq j\leq n-1$} \; ,
\\ V_0(\mcL^n) = V_n(\mcL^n) & = \frac{\binom{(n-2)/2}{-1/2}}{2^{n/2}}\cdot 
       \HG\left(1,\tfrac{1}{2};\tfrac{n+1}{2};-1\right) \quad 
       \left[ \sim \frac{\binom{(n-2)/2}{-1/2}}{2^{n/2}} \;,
       \text{ for $n\to\infty$}\right] \; .
\end{align*}
where $\binom{x}{y}=\frac{\Gamma(x+1)}{\Gamma(y+1)\cdot \Gamma(x-y+1)}$, and $\HG$ denotes the ordinary hypergeometric function (cf.~\cite[Ch.~15]{AbrSteg}). Furthermore, we have for direct products the simple convolution rule stated in Proposition~\ref{prop:facts-intrvol}\eqref{enum:prod-rule}. With Theorem~\ref{thm:main} we have thus formulas for the intrinsic volumes of all symmetric cones, which do not have the exceptional $27$-dimensional cone as one of its components.
\end{remark}

It is interesting to see how the formulas for the curvature measures in Theorem~\ref{thm:main} fit into the well-known framework of random matrices from the Gaussian Orthogonal/Unitary/Symplectic Ensemble. 
In the following remark we provide such a connection, which may be interpreted as a ``sanity check'' of the formula for $\Phi_j(\be,n,r)$ given in the main Theorem~\ref{thm:main}.

\begin{remark}
It is easily seen that the projection map $\Pi_{\Cbn}\colon\Herbn\to\Cbn$ simply replaces the negative eigenvalues of a matrix $A\in\Herbn$ by~$0$. It follows that a full rank matrix $A\in\Herbn$ has exactly $r$ positive eigenvalues iff the projection $\Pi_{\Cbn}(A)$ lies in $\Wbnr$, cf.~\eqref{eq:decomp-C-rank}. Using Proposition~\ref{prop:facts-intrvol}\eqref{enum:prob-proj}, we obtain that the probability that a matrix~$A$ from the $\GbE$ has exactly $r$ positive eigenvalues is given by
\begin{equation}\label{eq:prob-r-pos-eval}
  \Prob\big[ A \text{ has exactly $r$ positive eigenvalues}\big] = \sum_{j=0}^{d_{\be,n}} \Phi_j(\be,n,r) \; .
\end{equation}
On the other hand, using the formula~\eqref{eq:distr-eval-GbE} for the distribution of the joint probability density function for the eigenvalues of matrices from the $\GbE$, we obtain
\begin{equation}\label{eq:Prob[pos/neg_evals]}
  \Prob\big[ A \text{ has exactly $r$ pos.~eigenvalues}\big] = \frac{1}{(2\pi)^{n/2}\cdot F_n(\be/2)}\cdot \binom{n}{r} \cdot \underset{\IR_+^r\times \IR_-^{n-r}}{\int} e^{-\frac{\|z\|^2}{2}} \cdot |\Delta(z)|^\be\;dz ,
\end{equation}
where $\IR_-^k:=-\IR_+^k$. And indeed,~\eqref{eq:Prob[pos/neg_evals]} coincides with~\eqref{eq:prob-r-pos-eval}, as is seen by the following computation
\begin{align*}
   \eqref{eq:Prob[pos/neg_evals]} & \stackrel{\eqref{eq:decomp-Vanderm-f}}{=} \frac{1}{(2\pi)^{n/2}\cdot F_n(\be/2)}\cdot \binom{n}{r} \cdot \underset{\IR_+^r\times \IR_-^{n-r}}{\int} e^{-\frac{\|z\|^2}{2}} \cdot \bigg|\Delta(x)^\be\cdot \Delta(y)^\be\cdot \sum_{j=0}^{\be r(n-r)} f_{\be,j}(x;-y)\bigg|\;dz
\\ & = \frac{1}{(2\pi)^{n/2}\cdot F_n(\be/2)}\cdot \binom{n}{r} \cdot \underset{\IR_+^n}{\int} e^{-\frac{\|z\|^2}{2}} \cdot |\Delta(x)|^\be\cdot |\Delta(y)|^\be\cdot \sum_{j=0}^{\be r(n-r)} f_{\be,j}(x;y) \;dz
\\ & \stackrel{\eqref{eq:def-J}}{=} \sum_{j=0}^{\be r(n-r)} \binom{n}{r} \cdot \frac{J_\be(n,r,j)}{F_n(\be/2)} = \sum_{j=0}^{d_{\be,n}} \binom{n}{r} \cdot \frac{J_\be(n,r,j-d_{\be,r})}{F_n(\be/2)} \stackrel{\eqref{eq:Phi_j(C,M)}}{=} \sum_{j=0}^{d_{\be,n}} \Phi_j(\be,n,r) \; .
\end{align*}
\end{remark}

\subsection{Examples in small dimensions}\label{sec:ex-small-dims}

In this section we give the values of the curvature measures $\Phi_j(\be,n,r)$
and the intrinsic volumes of $\C_{\be,n}$ for dimension $n=1,2,3$.

The case $n=1$ is of course trivial, and we only mention it for the sake of completeness: we have $\Phi_0(\be,1,0)=V_0(\C_{\be,1})=\Phi_1(\be,1,1)=V_1(\C_{\be,1})=\frac{1}{2}$, cf.~\eqref{eq:J_b(1,0,0)=J_b(1,1,0)=...}. The remaining values for $n=1$ are zero.

As for the nontrivial cases $n=2,3$, recall that we have the symmetry relations 
$\Phi_j(\be,n,r) = \Phi_{d_{\be,n}-j}(\be,n,n-r)$ and $V_j(C_{\be,n}) = V_{d_{\be,n}-j}(C_{\be,n})$, 
cf.~\eqref{eq:symm-Phi_j}.
This halves the number of values we have to determine. For $n=2$ the $J$-integrals are easily computed with any computer algebra system. We obtain for the curvature measures
\begin{align*}
   \textstyle \Phi_0(1,2,0) & \textstyle = \frac{1}{2}-\frac{\sqrt{2}}{4} \;, & \textstyle \Phi_1(1,2,1) & \textstyle = \frac{\sqrt{2}}{4}
\\ \textstyle \Phi_0(2,2,0) & \textstyle = \frac{1}{4}-\frac{1}{2\pi} \;, & \textstyle \Phi_1(2,2,1) & \textstyle = \frac{1}{4} \;, & \textstyle \Phi_2(2,2,1) & \textstyle = \frac{1}{\pi}
\\ \textstyle \Phi_0(4,2,0) & \textstyle = \frac{1}{4}-\frac{2}{3\pi} \;, & \textstyle \Phi_1(4,2,1) & \textstyle = \frac{1}{8} \;, & \textstyle \Phi_2(4,2,1) & \textstyle = \frac{2}{3\pi} \;, & \textstyle \Phi_3(4,2,1) & \textstyle = \frac{1}{4}
\end{align*}
(the remaining values for $n=2$ are either~$0$, cf.~\eqref{eq:Pat-ineq_Phi}, or they are obtained from the above values via the symmetry~\eqref{eq:symm-Phi_j}).

As for the case $n=3$, the $J$-integrals $J_\be(3,r,k)$ for $r\in\{1,2\}$ are easily computed with any computer algebra system. We then obtain the value $J_\be(3,0,0)=J_\be(3,3,0)$ by first computing $V_1(C_{\be,3}),\ldots,V_{d_{\be,3}-1}(C_{\be,3})$, which only depend on $J_\be(3,r,k)$ for $r\in\{1,2\}$. By the self-duality of~$\Cbn$ we have $V_0(C_{\be,3})=V_{d_{\be,3}}(C_{\be,3})$. Combining this with Proposition~\ref{prop:facts-intrvol}\eqref{enum:ivol-sum=1} we obtain $2 V_0(C_{\be,3}) = 1-(V_1(C_{\be,3})-V_2(C_{\be,3})-\ldots-V_{d_{\be,3}-1}(C_{\be,3}))$. From Theorem~\ref{thm:main} we thus obtain
\begin{align*}
   J_\be(3,0,0) & = F_3(\be/2)\cdot V_0(C_{\be,3})
\\ & = F_3(\be/2)\cdot \tfrac{1}{2}\cdot \big(1-V_1(C_{\be,3})-V_2(C_{\be,3})-\ldots-V_{d_{\be,3}-1}(C_{\be,3})\big) \; .
\end{align*}
The resulting values of the curvature measures are
  \[ \textstyle \Phi_0(1,3,0) = \frac{1}{4}-\frac{\sqrt{2}}{2\pi} \;,\qquad \Phi_0(2,3,0) = \frac{1}{8}-\frac{3}{8\pi} \;,\qquad \Phi_0(4,3,0) = \frac{1}{8}-\frac{47}{120\pi} \; , \]
and the values 
$\Phi_j(\be,3,1)$ are given in 
Table~\ref{tab:Phi_j(b,3,1)}. (The remaining values for $n=3$ are obtained via the symmetry~\eqref{eq:symm-Phi_j}).
%
\begin{table}
  \[ \begin{array}{c||c|c|c|c|c|c|c|c|c}
        \tikz[baseline=1.5mm]{\draw (0,0.5) node[below=4mm, right=0mm]{$\be$} -- (1,0) node[above=4mm, left=0mm]{$j$};} & 1 & 2 & 3 & 4 & 5 & 6 & 7 & 8 & 9
     \\\hline & & & & & & & & & \\[-4.6mm]\hline\rule{0mm}{5mm}1 & \frac{\sqrt{2}}{4}-\frac{1}{4} & \frac{\sqrt{2}}{2\pi} & \frac{1}{2}-\frac{\sqrt{2}}{4} & \scriptstyle0 & \scriptstyle0 & \scriptstyle0 & \scriptstyle0 & \scriptstyle0 & \scriptstyle0
     \\[2mm]\hline\rule{0mm}{5mm} 2 & \frac{3}{16}-\frac{1}{2\pi} & \frac{1}{4\pi} & \frac{1}{2\pi} & \frac{1}{2\pi} & \frac{3}{16}-\frac{3}{8\pi} & \scriptstyle0 & \scriptstyle0 & \scriptstyle0 & \scriptstyle0
     \\[2mm]\hline\rule{0mm}{5mm} 4 & \frac{11}{64}-\frac{8}{15\pi} & \frac{1}{40\pi} & \frac{4}{15\pi}-\frac{1}{16} & \frac{19}{120\pi} & \frac{3}{32} & \frac{2}{5\pi} & \frac{1}{16}+\frac{1}{6\pi} & \frac{1}{5\pi} & \frac{7}{64}-\frac{7}{24\pi}
     \end{array}\]
  \caption{The values of $\Phi_j(\be,3,1)$.}
  \label{tab:Phi_j(b,3,1)}
\end{table}
%

The values of the intrinsic volumes of $\C_{\be,n}$, $n=1,2,3$, are summarized in Table~\ref{tab:V(C_(b,n))}.
\begin{table}
  \[ \begin{array}{c||c|c|c|c|c|c|c|c|c}
        & V_0 & V_1 & V_2 & V_3 & V_4 & V_5 & V_6 & V_7 & V_8
     \\\hline & & & & & & & & & \\[-4.6mm]\hline\rule{0mm}{5mm}
        \C_{\be,1} & \frac{1}{2} & \frac{1}{2} & \scriptstyle0 & \scriptstyle0 & \scriptstyle0 & \scriptstyle0 & \scriptstyle0 & \scriptstyle0 & \scriptstyle0
     \\[2mm]\hline\rule{0mm}{5mm}
        \C_{1,2} & \frac{1}{2}-\frac{\sqrt{2}}{4} & \frac{\sqrt{2}}{4} & \frac{\sqrt{2}}{4} & \frac{1}{2}-\frac{\sqrt{2}}{4} & \scriptstyle0 & \scriptstyle0 & \scriptstyle0 & \scriptstyle0 & \scriptstyle0
     \\[2mm]\hline\rule{0mm}{5mm}
        \C_{2,2} & \frac{1}{4}-\frac{1}{2\pi} & \frac{1}{4} & \frac{1}{\pi} & \frac{1}{4} & \frac{1}{4}-\frac{1}{2\pi} & \scriptstyle0 & \scriptstyle0 & \scriptstyle0 & \scriptstyle0
     \\[2mm]\hline\rule{0mm}{5mm}
        \C_{4,2} & \frac{1}{4}-\frac{2}{3 \pi} & \frac{1}{8} & \frac{2}{3\pi} & \frac{1}{4} & \frac{2}{3\pi} & \frac{1}{8} & \frac{1}{4}-\frac{2}{3 \pi} & \scriptstyle0 & \scriptstyle0
     \\[2mm]\hline\rule{0mm}{5mm}
        \C_{1,3} & \frac{1}{4}-\frac{\sqrt{2}}{2\pi} & \frac{\sqrt{2}}{4}-\frac{1}{4} & \frac{\sqrt{2}}{2\pi} & 1-\frac{\sqrt{2}}{2} & \frac{\sqrt{2}}{2\pi} & \frac{\sqrt{2}}{4}-\frac{1}{4} & \frac{1}{4}-\frac{\sqrt{2}}{2\pi} & \scriptstyle0 & \scriptstyle0
     \\[2mm]\hline\rule{0mm}{5mm}
        \C_{2,3} & \frac{1}{8}-\frac{3}{8\pi} & \frac{3}{16}-\frac{1}{2\pi} & \frac{1}{4\pi} & \frac{1}{2\pi} & \frac{3}{16}+\frac{1}{8\pi} & \frac{3}{16}+\frac{1}{8\pi} & \frac{1}{2\pi} & \frac{1}{4\pi} & \ldots
     \\[2mm]\hline\rule{0mm}{5mm}
        \C_{4,3} & \frac{1}{8}-\frac{47}{120\pi} & \frac{11}{64}-\frac{8}{15\pi} & \frac{1}{40\pi} & \frac{4}{15\pi}-\frac{1}{16} & \frac{19}{120\pi} & \frac{3}{32} & \frac{13}{120\pi}+\frac{7}{64} & \frac{11}{30\pi}+\frac{1}{16} & \ldots
     \end{array}\]
  \caption{Intrinsic volumes of $\C_{\be,n}$ for $n=1,2,3$ (the missing entries for $\C_{2,3}$ and $\C_{4,3}$ are obtained via $V_j(\C_{\be,n})=V_{d_{\be,n}-j}(\C_{\be,n})$).}
  \label{tab:V(C_(b,n))}
\end{table}

\section{Proof of the main result}\label{sec:proof-main-res}

In this section we provide the proof of the main Theorem~\ref{thm:main}. We first describe in Section~\ref{sec:intrvol-curv} how the intrinsic volumes and the curvature measures may be expressed in terms of curvature by stating Weyl's classical tube formula~\cite{weyl:39} and a generalization, which holds for a larger class of cones. In Section~\ref{sec:R-C-H} we state some general facts about the orthogonal/unitary/(compact) symplectic group, that we will use in Section~\ref{sec:comp} for the proof of Theorem~\ref{thm:main}.

\subsection{Expressing intrinsic volumes in terms of curvature}\label{sec:intrvol-curv}

From the characterizations~\eqref{eq:V_j(C)-polyhdrl} and~\eqref{eq:Phi_j(C,M)-polyhdrl} one easily obtains elementary formulas for the intrinsic volumes and for the curvature measures of polyhedral cones (although the actual computation of the intrinsic volumes may still very well be complicated as the resulting formulas include volumes of spherical polytopes).

Another class of cones, for which one has closed formulas for the intrinsic volumes, are \emph{smooth cones}, i.e., cones~$C\subseteq\IR^d$ such that the boundary $M:=\partial K$ of $K=C\cap S^{d-1}$ is a smooth (i.e., $C^\infty$) hypersurface of~$S^{d-1}$. The formulas for the intrinsic volumes involve the \emph{principal curvatures} of~$M$, which we shall describe next.


In general, let $M\subset S^{d-1}$ be a smooth submanifold of the unit sphere. For $p\in M$ we denote the tangent space of~$M$ in~$p$ by~$T_pM$, and we denote its orthogonal complement in $T_pS^{d-1}=p^\bot$ by $T^\bot_pM$. Let $\zeta\in T_pM$ be a tangent vector, and $\eta\in T^\bot_pM$ a normal vector. It can be shown that if $c\colon\IR\to M$ is a curve with $c(0)=p$ and $\dot{c}(0)=\zeta$, and if $w\colon\IR\to \IR^d$ is a normal extension of~$\eta$ along~$c$, i.e., $w(t)\in T^\bot_{c(t)}M$ and $w(0)=\eta$, then the orthogonal projection of $\dot{w}(0)$ onto $T_pM$ neither depends on the choice of the curve $c$ nor on the choice of the normal extension $w$ of $\eta$ (cf.~for example~\cite[Ch.~14]{thor:94} for the hypersurface case, or~\cite[Ch.~6]{dC} for general Riemannian manifolds). It therefore makes sense to define the map
  \[ W_{p,\eta}\colon T_pM\to T_pM \;,\quad \zeta\mapsto -\Pi_{T_pM}(\dot{w}(0)) \; , \]
where $w\colon\IR\to \IR^d$ is a normal extension of $\eta$ along a curve $c\colon\IR\to M$ which satisfies $c(0)=p$ and $\dot{c}(0)=\zeta$, and $\Pi_{T_pM}$ denotes the orthogonal projection onto the tangent space $T_pM$. This map is called the \emph{Weingarten map}.

It can be shown that $W_{p,\eta}$ is a symmetric linear map (cf.~\cite[Ch.~6]{dC}), so that it has $m:=\dim M$ real eigenvalues $\kappa_1(p,\eta),\ldots,\kappa_m(p,\eta)$, which are called the \emph{principal curvatures} of~$M$ at~$p$ in direction~$\eta$. The corresponding eigenvectors are called \emph{principal directions}. Furthermore, we denote the elementary symmetric functions in the principal curvatures by
\begin{align}\label{eq:def-sigma_i(p,eta)}
   \sigma_i(p,\eta) & := \sum_{1\leq j_1<\ldots<j_i\leq m} \kappa_{j_1}(p,\eta)\cdots\kappa_{j_i}(p,\eta) \; .
\end{align}

When we are working with orientable hypersurfaces, i.e., with submanifolds of codimension~$1$, which are endowed with a global unit normal vector field~$\nu\colon M\to T^\bot M$, $\nu(p)\in T^\bot_pM$, $\|\nu(p)\|=1$, then we abbreviate $\sigma_i(p):=\sigma_i(p,\nu(p))$. When $M=\partial K$ is the boundary of a spherically convex set, and additionally a smooth hypersurface of $S^{d-1}$, then we always consider $M$ to be endowed with the unit normal field pointing \emph{inwards} the set~$K$ (this implies $\kappa_i(p)\geq0$ for all $i=1,\ldots,d-2$).

In the context of (spherically) convex sets, Weyl's classical tube formula~\cite{weyl:39} says the following: Let $C\subseteq\IR^d$ be a closed convex cone such that the boundary $M=\partial K$ of $K=C\cap S^{d-1}$ is a smooth hypersurface of~$S^{d-1}$. Then, for $1\leq j\leq d-1$,
\begin{equation}\label{eq:tube-form_weyl}
  V_j(C) = \frac{1}{\mO_{j-1}\cdot \mO_{d-j-1}}\cdot \underset{p\in M}{\int} \sigma_{d-j-1}(p)\,dM \; ,
\end{equation}
where $\mO_{d-1}:=\vol_{d-1}S^{d-1}=\frac{2\pi^{d/2}}{\Gamma(d/2)}$, and $dM$ denotes the volume element induced from the Riemannian metric on $M$.

The problem is that the cones~$\Cbn$, whose intrinsic volumes we want to compute, are neither polyhedral nor smooth (for $n\geq3$). But the rank decomposition~\eqref{eq:decomp-C-rank} yields a decomposition of~$\Cbn$ into smooth pieces, which is the basic idea behind the proof of Theorem~\ref{thm:main}. In the remainder of this section we define the notion of a \emph{stratifiable convex set}, which is a generalization of both polyhedral and smooth convex sets, and we state a suitable generalization of~\eqref{eq:tube-form_weyl}.

%
%
In the following let $M\subset S^{d-1}$ be a smooth submanifold of the unit sphere. We may consider the \emph{tangent} resp.~\emph{normal bundle of $M$} (cf.~\cite[Ch.~3]{spiv1}) as submanifolds of $\IR^d\times\IR^d$ via
  \[ TM = \bigcup_{p\in M} \{p\}\times T_pM \;,\qquad T^\bot M = \bigcup_{p\in M} \{p\}\times T^\bot_pM \; . \]
Furthermore, 
if $M\subseteq S^{d-1}$, we also consider the \emph{spherical normal bundle}
\begin{align}
   T^S M & := \bigcup_{p\in M} \{p\}\times T_p^S M \;,\qquad T_p^S M := T^\bot_pM\cap S^{d-1} \; . \label{eq:def-spher-norm-bdl}
\end{align}
%

The tangent and the normal bundle are both so-called \emph{vector bundles}, as all fibers of the canonical projection maps $(x,v)\mapsto x$ are vector spaces.
The spherical normal bundle is a \emph{sphere bundle}, as all fibers are subspheres of the unit sphere. For the generalization of Weyl's tube formula we need to consider another class of fiber bundles, where each fiber is given by (the relative interior of) a spherically convex set.

Let $C\subseteq\IR^d$ be a closed convex cone, and let $\Pi_C$ denote the canonical projection onto~$C$. For $p\in C$ we define the \emph{normal cone} of~$C$ in~$p$ by
  \[ N_p(C) := \{ v\in\IR^d \mid \Pi_C(v+p)=p \} \; , \]
which is easily seen to be a closed convex cone with $N_p(C)\subseteq p^\bot$. For a subset $M\subseteq C$, we define the 
\emph{spherical duality bundle} via
\begin{equation}\label{eq:def-N^SM}
  N^SM := \bigcup_{p\in M} \{p\}\times N_p^S M \;,\qquad N_p^S M := \relint(N_p(C))\cap S^{d-1} \; .
\end{equation}
Note that we have not imposed any smoothness assumption yet, but if $M\subseteq C\cap S^{d-1}$ is smooth, then we have $N^SM\subseteq T^SM$. Note also that $N^SM$ in fact depends on~$M$ and~$C$.

\begin{definition}\label{def:stratified}
Let $C\subseteq\IR^d$ be a closed convex cone. We call the spherically convex set $K:=C\cap S^{d-1}$ \emph{stratifiable} if it decomposes into a disjoint union~$K = \dot{\bigcup}_{i=0}^k M_i$, such that:
\begin{enumerate}
  \item For all $0\leq i\leq k$, $M_i$ is a smooth connected submanifold of $S^{d-1}$.
  \item For all $0\leq i\leq k$ the spherical duality bundle $N^SM_i$ is a smooth manifold.
\end{enumerate}
If (1) and (2) are satisfied, then we call~$K = \dot{\bigcup}_{i=0}^k M_i$ a \emph{valid decomposition}. Furthermore, we call a stratum~$M_i$ \emph{essential} if $\dim N^SM_i=d-2$, otherwise we call it \emph{negligible}.
\end{definition}

The following theorem is the announced generalization of Weyl's tube formula~\eqref{eq:tube-form_weyl} to stratified sets. A proof may be found in~\cite[\S4.3]{am:thesis}. Formulas similar to the one that we give in the following theorem may also be found in~\cite{AT:07}.

\begin{thm}\label{thm:form-intr-vol}
Let $C\subseteq\IR^d$ such that $K:=C\cap S^{d-1}$ is stratifiable and decomposes into the valid decomposition $K = \dot{\bigcup}_{i=0}^{\tilde{k}} M_i$, with $M_0=\inter(K)$, and $M_1,\ldots,M_k$ denoting the essential and $M_{k+1},\ldots,M_{\tilde{k}}$, $k\leq \tilde{k}$, denoting the negligible pieces. Then, for $1\leq j\leq d-1$,
\begin{align}
   V_j(C) & \;=\; \sum_{i=1}^k \Phi_j(C,M_i) \; ,
\label{eq:V_j(C)-strat}
\\ \Phi_j(C,M_i) & \;=\; \frac{1}{\mO_{j-1}\cdot\mO_{d-j-1}}\cdot \underset{p\in M_i}{\int} \; \underset{\eta\in N_p^S(C)}{\int} \sigma^{(i)}_{d_i-j-1}(p,-\eta) \;dN_p^S(C)\;dM_i \,,\; \text{for } i=1,\ldots k \; ,
\label{eq:Phi_j(C,M_i)-strat}
\end{align}
where $d_i:=\dim M_i+2$, $\sigma^{(i)}_\ell(p,-\eta)$ denotes the $\ell$th elementary symmetric function in the principal curvatures of $M_i$ at $p$ in direction $-\eta$, and $\sigma_\ell(p,-\eta):=0$ if $\ell<0$.
\end{thm}

A paraphrase of~\eqref{eq:V_j(C)-strat} is that the curvature measures vanish at the negligible pieces. Note also that in the sum~\eqref{eq:V_j(C)-strat} only those strata~$M_i$ contribute, for which $j\leq d_i-1$.

\subsection{Orthogonal, unitary, (compact) symplectic group}\label{sec:R-C-H}

In this section we discuss the compact Lie group, which preserves the canonical scalar product on~$\IF_\be^n$ given by $\langle x,y\rangle = x^\dagger y = \sum_{i=1}^n \bar{x}_i y_i$, $x,y\in\IF_\be^n$. We denote this group by
\begin{align*}
   G(n) := G_\be(n) & := \{ U\in\IF_\be^{n\times n}\mid \forall x,y\in\IF_\be^n : \langle Ux,Uy\rangle = \langle x,y\rangle \}
\\ & = \{ U\in\IF_\be^{n\times n}\mid U^\dagger U = I_n \} \; .
\end{align*}
For $\be=1,2,4$ the group $G_\be(n)$ is called the \emph{orthogonal group}, \emph{unitary group}, and \emph{(compact) symplectic group}, respectively, cf.~for example~\cite[\S7.2]{FH:91}. Note that an element $U\in G_\be(n)$ may be identified with an orthonormal basis of $\IF_\be^n$ by interpreting the matrix~$U$ as the $n$-tuple of its columns. We drop the index $\beta$ to simplify the notation.

The Lie algebra of $G(n)$, which we identify with the tangent space of $G(n)$ at the identity matrix~$I_n$, is given by the real vector space of skew-Hermitian matrices
  \[ \Skew_n := \Skew_{\be,n} := T_{I_n} G(n) = \{ A\in\IF_\be^{n\times n}\mid A^\dagger=-A \} \; . \]
To specify a left-invariant Riemannian metric on $G(n)$ it suffices to declare an $\IR$-basis of the Lie algebra~$\Skew_n$ to be orthonormal (and then extend the metric to $G(n)$ by pushing it forward via the left-multiplication). For $\be=4$ we declare the following basis of~$\Skew_n$ to be orthonormal:
\begin{align}\label{eq:basis-Skew_n}
   \{ \iota E_{ii} & \mid 1\leq i\leq n, \iota\in\{\bi,\bj,\bk\}\}
\\ & \cup \{ E_{ij}-E_{ji}\mid 1\leq j<i \leq n\}\cup \{ \iota(E_{ij}+E_{ji})\mid 1\leq j<i \leq n, \iota\in\{\bi,\bj,\bk\}\} \; , \nonumber
\end{align}
and for $\be=1,2$ we use its intersections with $\IR^{n\times n}$ and $\IC^{n\times n}$, respectively. It is readily checked that this yields a bi-invariant metric on~$G(n)$ (the bi-invariance in fact determines the Riemannian metric up to scaling).

Before we calculate the volume of $G(n)$ with respect to the volume determined by the above Riemannian metric, we recall the well-known \emph{(smooth) coarea formula}. The coarea formula is essential for the computation in Section~\ref{sec:comp}.

If $L\colon V\to W$ is a surjective linear operator 
between euclidean vector spaces $V$ and $W$, then we define the \emph{normal determinant} of $L$ as
  \[ \ndet(L) := |\det(L|_{\ker(L)^\bot})| \; , \]
where $L|_{\ker(L)^\bot}\colon \ker(L)^\bot\to W$ denotes the restriction of $L$ to the orthogonal complement of the kernel of $L$. Obviously, if 
$L$ is a bijective linear operator, then $\ndet(L)=|\det(L)|$, so the normal determinant provides a natural generalization of the absolute value of the determinant.

\begin{lemma}\label{lem:coarea}
Let $\vp\colon \M_1\to \M_2$ be a smooth surjective map between Riemannian 
manifolds~$\M_1,\M_2$. Then for any $f\colon \M_1\to\IR$ that is integrable w.r.t.~$d\M_1$ we have
\begin{align}
   \underset{\M_1}{\int} f\,d\M_1 & = \underset{q\in \M_2}{\int}\;\underset{p\in \vp^{-1}(q)}{\int} 
                \frac{f(p)}{\ndet(D_p\vp)} \,d\vp^{-1}(q)\,d\M_2 \; . 
                \label{eq:cor-coarea-Riem-1}
\end{align}
\end{lemma}

The inner integral in~\eqref{eq:cor-coarea-Riem-1} over the fiber $\vp^{-1}(q)$ is well-defined for almost all $q\in \M_2$. This follows from Sard's lemma (cf.~for example~\cite[Thm.~3-14]{spiv:65}), which implies that almost all $q\in \M_2$ are regular values, i.e., the derivative $D_p\vp$ has full rank for all $p\in\vp^{-1}(q)$. The fibers $\vp^{-1}(q)$ of regular values $q$ are smooth submanifolds of $\M_1$ and therefore the integral over $\vp^{-1}(q)$ is well-defined.
One calls $\ndet(D_p\vp)$ the \emph{Normal Jacobian} of~$\vp$ at~$p$.

See~\cite[3.8]{Mor:95} or~\cite[3.2.11]{Fed} for proofs of the coarea formula with $\M_1,\M_2$ 
being submanifolds of euclidean space. See~\cite[Appendix]{H:93} for a proof of the coarea formula in the above stated form.

To get the volume of $G(n)$ we consider its action on $\IF_\be^n$. Note that we have an inner product on $\IF_\be^n$ given by $\langle x,y\rangle = x^\dagger y$. Considering $\IF_\be^n$ as a $(\be n)$-dimensional real vector space, we have the inner product on $\IF_\be^n$ given by $\langle x,y\rangle_\IR:=\Re(\langle x,y\rangle)$, where $\Re\colon\IF_\be\to\IR$ denotes the canonical projection. For $\be=4$ this amounts to the same as declaring the set $\{ \iota\cdot e_i\mid \iota\in\{1,\bi,\bj,\bk\}\}$ an orthonormal basis of the real $(4n)$-dimensional vector space $\IH^n$, where $e_i\in\IR^n$ denotes the $i$th canonical basis vector.

We now get the volume of $G(n)$ in the following way: The map $\vp\colon G(n)\to S(\IF_\be^n)=\{x\in\IF_\be^n\mid \|x\|=1\}$, $U\mapsto U\cdot e_1$, is a \emph{Riemannian submersion}, i.e., for every $U\in G(n)$ the restriction of $D_U\vp$ to the orthogonal complement of its kernel is an isometry. In particular, the Normal Jacobian of $\vp$ is everywhere equal to~$1$. Furthermore, each fiber $\vp^{-1}(x)$ is isometric to $G(n-1)$ as is easily checked. An application of the coarea formula thus yields $\vol G(n)=\vol S(\IF_\be^n)\cdot \vol G(n-1)$. As $\vol S(\IF_\be^n) = \vol S^{\be n-1} = \mO_{\be n-1}$, we obtain by induction
\begin{equation}\label{eq:vol(G(n))}
   \vol G(n) = \prod_{i=1}^n \mO_{\be i-1} = \prod_{i=1}^n \frac{2\pi^{\be i/2}}{\Gamma(\frac{\be i}{2})} = 2^n\cdot \pi^{n(n+1)\be/4}\cdot \prod_{i=1}^n\frac{1}{\Gamma(\frac{\be i}{2})} \; .
\end{equation}

In Section~\ref{sec:comp} we will need to consider certain subgroups of $G(n)$. By a \emph{distribution} of $r\in\IZ_{>0}$ we understand a tuple $\rho=(\rho_1,\ldots,\rho_m)\in\IZ_{>0}^m$ such that $|\rho|:=\rho_1+\ldots+\rho_m=r$. For such $\rho$ with $|\rho|\leq n$ we define the closed subgroup $G(n,\rho)$ of $G(n)$ consisting of the matrices having a block-diagonal form prescribed by~$\rho$:
\begin{equation}\label{eq:G(n,rho)}
   G(n,\rho) := \left\{ \diag(U_1,\ldots,U_m,U') \mid U_i\in G(\rho_i) , U'\in G(n-r) \right\} \; .
\end{equation}
Note that $G(n,\rho)$ with its induced Riemannian metric is isometric to the direct product $G(\rho_1)\times\ldots\times G(\rho_m)\times G(n-r)$. Furthermore, the homogeneous space $G(n)/G(n,\rho)$ is a smooth manifold of dimension
\begin{align}\label{eq:dim(G(n)/G(n,rho))}
   \dim G(n)/G(n,\rho) & = \dim G(n) - \dim G(n-r) - \sum_{i=1}^m \dim G(\rho_i) = d_{\be,n} - d_{\be,n-r} - \sum_{i=1}^m d_{\be,\rho_i} \nonumber
\\ & = \be \left(\binom{n}{2} - \binom{n-r}{2} - \sum_{i=1}^m \binom{\rho_i}{2} \right) \; .
\end{align}
The case $\rho=1^{(r)}=(1,\ldots,1)$ ($r$-times) will be of particular importance. Note that $G(1)=S(\IF_\be)=\{a\in\IF_\be\mid \|a\|=1\}$, so that $G(n,1^{(r)}) \cong S(\IF_\be) \times \ldots \times S(\IF_\be) \times G(n-r)$. We use the notation
\begin{equation}\label{eq:def-G_(n,r)}
  G_{n,r} := G(n)/G(n,1^{(r)}) \; .
\end{equation}
Furthermore, we denote by $G(n)\to G_{n,r}$, $U\mapsto[U]:=U\cdot G(n,1^{(r)})$ the canonical projection.

Note that $G(n)$ has a natural action on $G_{n,r}$ given by $(U_1,[U_2])\mapsto [U_1U_2]$ for $U_1,U_2\in G(n)$. Moreover, as $G(n)$ acts transitively on $G_{n,r}$, there exists up to scaling at most one Riemannian metric on $G_{n,r}$, which is $G(n)$-invariant. In the following paragraphs we will give a concrete description of the tangent space $T_{[I_n]}G_{n,r}$, and specify on it a $G(n)$-invariant Riemannian metric.

As $G(n,1^{(r)})=\left\{\left.\begin{pmatrix} \Lambda & 0 \\ 0 & U'\end{pmatrix}\right| \Lambda=\diag(\lambda_1,\ldots,\lambda_r), \lambda_i\in S(\IF_\be), U'\in G(n-r)\right\}$, the tangent space of $G_{n,r}$ at $[I_n]$ is given by
  \[ T_{I_n}G(n,1^{(r)}) = \left\{ \left. \begin{pmatrix} D & 0 
               \\ 0 & S \end{pmatrix} \right| D=\diag(a_1,\ldots,a_r) \,,\; \Re(a_i)=0 \,,\; S^\dagger=-S\right\} \; . \]
The orthogonal complement of $T_{I_n}G(n,1^{(r)})$ in $T_{I_n}G(n)=\Skew_n$, the space of skew-Hermitian matrices, is given by
\begin{equation}\label{eq:bar(Skew_n)}
  \ol{\Skew_n} := (T_{I_n} G(n,1^{(r)}))^\bot = \left\{ \left. \begin{pmatrix} X & -Y^\dagger \\ Y & 0 \end{pmatrix} \right| X\in\IF_\be^{r\times r}, X^\dagger=-X, Y\in \IF_\be^{(n-r)\times r} \right\} \; .
\end{equation}

It can be shown (cf.~\cite[Lemma~II.4.1]{helga}) that there exists an open ball~$B$ around the origin in~$T_{I_n}G(n)=\Skew_n$ such that the intersection $B\cap \ol{\Skew_n}$ is diffeomorphic to an open neighborhood of~$[I_n]$ in~$G_{n,r}$. Moreover, the tangent space of~$G_{n,r}$ in~$[I_n]$ may be identified with $\ol{\Skew_n}$, and the restriction of the inner product on $\Skew_n$ to $\ol{\Skew_n}$ yields a well-defined Riemannian metric on $G_{n,r}$, which is $G(n)$-invariant. (See~\cite[\S5.2]{ambu:11b} for a more detailed description of the induced Riemannian metric on a homogeneous space in a similar situation.)

For $\be=4$ we have the following orthonormal basis of $\ol{\Skew_n}$, cf.~\eqref{eq:basis-Skew_n},
  \[ \{ E_{ij}-E_{ji} \mid (i,j)\in \mI \}\cup \{ \iota (E_{ij} + E_{ji})\mid \iota\in \{\bi,\bj,\bk\}, (i,j)\in \mI \} \; , \]
where $\mI := \mI_1\cup\mI_2$ with
\begin{equation}\label{eq:def-I1,I2}
  \mI_1 := \{ (i,j)\mid 1\leq j<i\leq r \} \,,\quad \mI_2 := \{ (i,j)\mid r+1\leq i\leq n \,,\; 1\leq j\leq r \} \;.\quad
\begin{array}[c]{c}
\begin{tikzpicture}[scale=0.6]
\fill[gray!20!white] (0,-0.3) -- (0,2) -- (1,1) -- (1,-0.3) -- cycle;
\draw (0,-0.3) rectangle (2.3,2)
      (0,2) -- (2.3,-0.3)
      (0,1) -- (1,1) -- (1,-0.3);
\draw[dashed] (1,2) -- (1,1) -- (2.3,1);
\path (0.4,1.3) node{$\scriptstyle \mI_1$} (0.5,0.35) node{$\scriptstyle \mI_2$};
\end{tikzpicture}
\end{array}
\end{equation}
For further use in Section~\ref{sec:comp}, we denote this orthonormal basis of $\ol{\Skew_n}\cong T_{[I_n]}G_{n,r}$ (for $\be=4$) by
\begin{equation}\label{eq:def-eta_(ij)}
  \eta_{ij}^1 := E_{ij}-E_{ji} \;,\qquad \eta_{ij}^\iota := \iota(E_{ij}+E_{ji}) \;,\qquad (i,j)\in\mI \,,\; \iota\in \{\bi,\bj,\bk\} \; .
\end{equation}
For $\be=2$ we have the orthonormal basis $\{\eta_{ij}^\iota\mid \iota\in\{1,\bi\}, (i,j)\in\mI\}$, and for $\be=1$ we have the orthonormal basis $\{\eta_{ij}^1\mid (i,j)\in\mI\}$.

The canonical projection $G(n)\to G_{n,r}$ turns out to be a Riemannian submersion, and an application of the coarea formula yields
\begin{equation}\label{eq:vol(G_(n,r))}
   \vol G_{n,r} = \frac{\vol G(n)}{\vol G(n,1^{(r)})} = \frac{\vol G(n)}{(\vol S(\IF_\be))^r\cdot \vol G(n-r)} = \frac{\vol G(n)}{\mO_{\be-1}^r\cdot \vol G(n-r)} \; .
\end{equation}

\subsection{Deducing the formulas for \texorpdfstring{$\Phi_j(\be,n,r)$}{Phi(beta,n,r)}}\label{sec:comp}

In this section we derive the claimed formulas for the curvature measures $\Phi_j(\be,n,r)$ stated in Theorem~\ref{thm:main}. The organization is straightforward and is as follows: First, we 
recall the well-known face structure of $\C_{\be,n}$, which is described for example in~\cite[\S II.12]{Barv}. Then we show that~$\C_{\be,n}$ is a stratified cone, and we determine the essential and the negligible pieces. Finally, we compute the principal curvatures of the strata, so that we can give the proof of Theorem~\ref{thm:main} by an application of Theorem~\ref{thm:form-intr-vol} and a small computation.

In this section we change to the spherical viewpoint and denote the intersection of $\C_{\be,n}$ 
with the unit sphere by
\begin{align*}
   K_n & := \C_{\be,n}\cap S(\Her_{\be,n}) & & \hspace{-27mm} = \{ A\in\Her_{\be,n}\mid A\seq0 \,,\; \|A\|=1 \} \; .
\end{align*}
First, we recall the well-known face structure of~$\C_{\be,n}$, cf.~Section~\ref{sec:intrvol-polyhedr}.

\begin{proposition}\label{prop:faces-SDP-cone}
The faces of $\C_{\be,n}$ are parametrized by the subspaces of $\IF_\be^n$. More precisely, for $L\subseteq\IF_\be^n$ an $\IF_\be$-linear subspace of $\IF_\be$-dimension~$r$, where $0\leq r\leq n-1$ the set $\{A\in \C_{\be,n}\mid \im(A)\subseteq L\}$ is a face of $\C_{\be,n}$ of dimension $d_{\be,r}$. Conversely, every face of $\C_{\be,n}$ is of this form and hence has dimension~$d_{\be,r}$ for some $0\leq r\leq n-1$. Moreover, every face of $\C_{\be,n}$ is of the form
\begin{equation}\label{eq:face-Hpd}
  \left\{ \left. U \begin{pmatrix} A' & 0 \\ 0 & 0 \end{pmatrix} U^\dagger \right| A'\in \C_{\be,r} \right\} \; ,
\end{equation}
where $U\in G_\be(n)$ and $0\leq r\leq n-1$. The normal cone at the face defined in~\eqref{eq:face-Hpd} is given by
\begin{equation}\label{eq:norm-face-Hpd}
  \left\{ \left. U \begin{pmatrix} 0 & 0 \\ 0 & -A'' \end{pmatrix} U^\dagger \right| A''\in \C_{\be,n-r} \right\} \; .
\end{equation}
\end{proposition}

\begin{proof}
See for example~\cite[\S II.12]{Barv}. The proof for the real case given there extends to the complex and the quaternion case in a straightforward way.
\end{proof}

Note that~\eqref{eq:face-Hpd} and~\eqref{eq:norm-face-Hpd} show that when analyzing a face of $\C_{\be,n}$, by choosing an appropriate basis of $\IF_\be^n$, we may assume without loss of generality that this face is of the form $\C_{\be,r}\times\{0\}$ with corresponding normal cone $\{0\}\times(-\C_{\be,n-r})$.

Next we will show that~$K_n$ is stratifiable and exhibit a valid decomposition of~$K_n$, cf.~Definition~\ref{def:stratified}. For this, we use a finer distinction than in~\eqref{eq:decomp-C-rank}: we classify the matrices $A\in K_n$ not only according to their ranks, i.e., according to the multiplicity of the eigenvalue~$0$, but also according to the multiplicities of the nonzero eigenvalues.
To achieve this we define the \emph{eigenvalue pattern} of an element~$A\in K_n$ via
  \[ \pa(A) := (\rho_1,\ldots,\rho_m) \;,\quad \text{iff } \lambda_1=\ldots=\lambda_{\rho_1} > \lambda_{\rho_1+1}=\ldots=\lambda_{\rho_1+\rho_2} >\ldots \; , \]
where $\lambda_1\geq\ldots\geq\lambda_r>0$ are the positive eigenvalues of~$A$.
Note that $\pa(A)$ is a distribution of $r=\rk(A)$. The spherical cap~$K_n$ thus decomposes into
\begin{equation}\label{eq:dec-ID_n}
  K_n = \dot{\bigcup}_{r=1}^n \dot{\bigcup}_{|\rho|=r} M_{n,\rho} \;,\qquad M_{n,\rho} := \{A\in K_n\mid \pa(A)=\rho \} \; .
\end{equation}
Note that $\inter(K_n) = \dot{\bigcup}_{|\rho|=n} M_{n,\rho}$ and $\partial K_n = \dot{\bigcup}_{r=1}^{n-1} \dot{\bigcup}_{|\rho|=r} M_{n,\rho}$.

\begin{proposition}\label{prop:strat-struct}
The set $M_{n,\rho}$, $|\rho|\leq n$, defined in~\eqref{eq:dec-ID_n} is a smooth submanifold of the unit sphere $S(\Herbn)$. Furthermore, the duality bundle $N^SM_{n,\rho}$ defined in~\eqref{eq:def-N^SM} is a smooth manifold for all $|\rho|\leq n$. Hence~\eqref{eq:dec-ID_n} is a valid decomposition.

The strata $\big\{ M_{n,1^{(r)}} \mid 1\leq r\leq n\big\}$, where $1^{(r)}:=(1,1,\ldots,1)$, are essential and all the other strata $M_{n,\rho}$ are negligible.
\end{proposition}

\begin{proof}
We fix a pattern $\rho=(\rho_1,\ldots,\rho_m)$ with $|\rho|=r\leq n$, and define the set $P\subset S^{n-1}$ via
\begin{equation}\label{eq:def-P_(n,rho)}
   P := \left\{\lambda\in S^{n-1}\left| \begin{array}{l}\lambda_1=\ldots=\lambda_{\rho_1}>\lambda_{\rho_1+1}=\ldots=\lambda_{\rho_1+\rho_2} > \ldots \\ \ldots \lambda_r > \lambda_{r+1}=\ldots=\lambda_n=0 \end{array} \right. \right\} \; .
\end{equation}
It is easily seen that the set $P$ is the intersection of the interior of an $m$-dimensional polyhedral cone with the unit sphere, so it is an $(m-1)$-dimensional submanifold of~$S^{n-1}$. We now consider the map
\begin{equation}\label{eq:def-psi_(n,rho)}
  \psi\colon P\times G(n)\to \Herbn \;,\quad (\lambda,U)\mapsto U\cdot \diag(\lambda)\cdot U^\dagger \; .
\end{equation}
This map is smooth, and its image is $M_{n,\rho}$ by the principle axis theorem. Concerning the fiber of $A\in M_{n,\rho}$, we may assume w.l.o.g.~that $A=\diag(\mu)$. Note that $U\cdot \diag(\lambda)\cdot U^\dagger = \diag(\mu)$ iff $\lambda=\mu \text{ and } U\cdot \diag(\lambda) = \diag(\lambda)\cdot U$. Furthermore, it is easily checked that for $\lambda\in P$, we have $U\cdot \diag(\lambda) = \diag(\lambda)\cdot U$ iff $U\in G(n,\rho)$.
We may thus conclude that, for any $A\in M_{n,\rho}$,
\begin{equation}\label{eq:psi_(n,rho)(lambda,Q)=A<=>...}
  \psi(\lambda,U) = A \iff \lambda=\mu \text{ and } U\in G(n,\rho) \; .
\end{equation}
This implies that the map $\psi$ factors over the product $P\times \big(G(n)/G(n,\rho)\big)$, i.e., we have a commutative diagram
\begin{equation}\label{eq:com-diag-psi_(n,rho)}
\begin{array}[c]{c}
\begin{tikzpicture}[>=stealth]
\matrix (m) [matrix of math nodes, row sep=2.5em, column sep=12mm, text height=2ex, text depth=0.5ex]
  {    P\times G(n) & \Herbn
    \\ P\times \big(G(n)/G(n,\rho)\big)
    \\ };
\path[->,font=\scriptsize]
  (m-1-1) edge node[auto]{$\psi$} (m-1-2);
\path[->,font=\scriptsize]
  (m-1-1) edge node[left]{$\Pi$} (m-2-1);
\path[->,dashed,font=\scriptsize]
  (m-2-1) edge node[below right]{$\ol{\psi}$} (m-1-2);
\end{tikzpicture}
\end{array} \; ,
\end{equation}
where $\Pi\colon P\times G(n)\to P\times \big(G(n)/G(n,\rho)\big)$ denotes the canonical projection map. Moreover,~\eqref{eq:psi_(n,rho)(lambda,Q)=A<=>...} implies that the map $\ol{\psi}$ is injective, and thus a bijection on the image of~$\psi$, which is~$M_{n,\rho}$. It is straightforward to show that the derivative of $\ol{\psi}$ has full rank. Since $P\times \big(G(n)/G(n,\rho)\big)$ is compact, it follows that $\ol{\psi}$ is a homeomorphism onto its image $M_{n,\rho}$. Hence $M_{n,\rho}$ is a submanifold of $\Herbn$ and $\ol{\psi}$ induces a diffeomorphism of $P\times \big(G(n)/G(n,\rho)\big)$ to~$M_{n,\rho}$, cf.~for example~\cite{booth}.

As for the claim about the duality bundle, note that for $A=U\cdot \diag(\lambda)\cdot U^\dagger$, $\lambda\in P$, the normal cone $N_A(K_n)$ is given by (cf.~Proposition~\ref{prop:faces-SDP-cone})
  \[ N_A(K_n) = \left\{ \left. U \begin{pmatrix} 0 & 0 \\ 0 & -A'' \end{pmatrix} U^\dagger \right| A''\in \C_{\be,n-r} \right\} \; . \]
Denoting by $\inter(K_{n-r})$ the interior of $K_{n-r}$ with respect to the topology on $S(\Her_{\be,n-r})$, we can define the map
\begin{align*}
   \Psi\colon P\times G(n)\times \inter(K_{n-r}) & \to \Herbn\times\Herbn \; ,
\\ (\lambda,U,A'') & \mapsto \left(\psi(\lambda,U) \,,\; U \begin{pmatrix} 0 & 0 \\ 0 & -A'' \end{pmatrix} U^\dagger\right) \; .
\end{align*}
This map is smooth, and its image is $N^S M_{n,\rho}$, cf.~\eqref{eq:def-N^SM} and~\eqref{eq:norm-face-Hpd}. If $\Psi(\lambda,U,A'')=\Psi(\mu,V,B'')$ then $\psi(\lambda,U)=\psi(\mu,V)$, i.e., $\lambda=\mu$ and $V=U\cdot D$ for some $D\in G(n,\rho)$. Furthermore, for the second component we get $U \begin{pmatrix} 0 & 0 \\ 0 & -A'' \end{pmatrix} U^\dagger = U D \begin{pmatrix} 0 & 0 \\ 0 & -B'' \end{pmatrix} D^\dagger U^\dagger$, i.e.,
\begin{equation}\label{eq:A''=D'.B''.(D')^dagger}
  A''=D' \cdot B''\cdot (D')^\dagger \; ,
\end{equation}
where $D=\diag(D_1,\ldots,D_m,D')$, cf.~\eqref{eq:G(n,rho)}. Using the conjugation~\eqref{eq:A''=D'.B''.(D')^dagger} to define an action of $G(n,\rho)$ on $\inter(K_{n-r})$, we may form the factor space $\big(G(n)\times \inter(K_{n-r})\big)/G(n,\rho)$. As the above arguments about the injectivity of $\Psi$ are reversible, it follows that $\Psi$ factors over the product $P\times \big(G(n)\times \inter(K_{n-r})\big)/G(n,\rho)$. We thus obtain the smooth injective map
  \[ \ol{\Psi}\colon P\times \big(G(n)\times \inter(K_{n-r})\big)/G(n,\rho) \to \Herbn\times\Herbn \; , \]
whose image is given by $N^SM_{n,\rho}$. Arguing as above shows that $N^SM_{n,\rho}$ is indeed a smooth manifold. Its dimension is given by
\begin{align*}
   \dim N^SM_{n,\rho} & = \dim P + \dim G(n) + \dim K_{n-r} - \dim G(n,\rho)
\\ & = (m-1) + d_{\be,n} + (d_{\be,n-r} - 1) - \Big( d_{\be,n-r} + \sum_{i=1}^m d_{\be,\rho_i}\Big)
\\ & = m-2 + n+\be \binom{n}{2} - \sum_{i=1}^m \Big(\rho_i+\be \binom{\rho_i}{2}\Big)
\\ & = n + \be\cdot \binom{n}{2} - 2 + m - r - \be\cdot \sum_{i=1}^m \binom{\rho_i}{2} \; .
\end{align*}
Note that since $m\leq r$, we have
  \[ m - r - \be\cdot \sum_{i=1}^m \binom{\rho_i}{2} \; \begin{cases} = 0 & \text{if } m=r \text{, i.e., } \rho=1^{(r)} \\ < 0 & \text{otherwise} \; . \end{cases} \]
Therefore, as $\dim \Herbn = d_{\be,n} = n+\be\cdot \binom{n}{2}$, we have
  \[ \dim N^SM_{n,\rho} \; \begin{cases} = \dim \Herbn - 2 & \text{if } \rho=1^{(r)} \\ < \dim \Herbn - 2 & \text{otherwise} \; . \end{cases} \]
Therefore, the strata $M_{n,1^{(r)}}$, $1\leq r\leq n$, are essential, and all other strata are negligible.
\end{proof}

From now on we may restrict the computations to the essential strata $M_{n,1^{(r)}}$. We use the notation
\begin{equation}\label{eq:def-P_r}
   P_r := \left\{\lambda\in S^{n-1} \mid \lambda_1>\lambda_2>\ldots>\lambda_r>0=\lambda_{r+1}=\ldots=\lambda_n\right\} \; .
\end{equation}
%
%
From the proof of Proposition~\ref{prop:strat-struct} we obtain the following corollary.

\begin{corollary}
Let $1\leq r\leq n$, and let $\vp_r$ be defined via
\begin{equation}\label{eq:def-phi_r}
  \vp_r\colon P_r\times G_{n,r} \to M_{n,1^{(r)}} \;,\quad (\lambda,[U])\mapsto U\cdot \diag(\lambda)\cdot U^\dagger \; .
\end{equation}
Then $\vp_r$ is well-defined and a diffeomorphism. \qed
\end{corollary}

In the following proposition we compute the Normal Jacobian of~$\vp_r$, i.e., the absolute value of its Jacobi determinant. This result is well-known, but we include the proof for the sake of completeness.

\begin{proposition}\label{prop:phi_(n,r)}
The Normal Jacobian of $\vp_r$ at $(\lambda,[U])$ is given by
  \[ \ndet(D_{(\lambda,[U])}\vp_r) = 2^{r(2n-r-1) \be/4}\cdot \Delta(\lambda)^\be\cdot \prod_{i=1}^r \lambda_i^{\be(n-r)} \; , \]
where $\Delta(\lambda)=\prod_{1\leq i<j\leq r} (\lambda_i-\lambda_j)$ denotes the Vandermonde determinant.
\end{proposition}

\begin{proof}
We restrict to the quaternion case $\be=4$; the cases $\be\in\{1,2\}$ are similar. By orthogonal invariance, we may assume w.l.o.g.~that $U=I_n$. For $\zeta\in T_\lambda P_r$ we have
\begin{equation}\label{eq:Dphi-lambda}
   D_{(\lambda,[I_n])}\vp_r(\zeta,0) = \diag(\zeta) \; .
\end{equation}

Recall that we identify $T_{[I_n]}G_{n,r}$ with $\ol{\Skew_n}$, which has the orthonormal basis $(\eta_{ij}^\iota)$ defined in~\eqref{eq:def-eta_(ij)}. Let $U_{ij}^\iota\colon\IR\to G(n)$, with $\iota\in\{1,\bi,\bj,\bk\}$ and $(i,j)\in\mI$ (cf.~\eqref{eq:def-I1,I2}), be curves such that the induced curves $[U_{ij}^\iota]\colon\IR\to G_{n,r}$ define the directions~$\eta_{ij}^\iota$, cf.~\eqref{eq:def-eta_(ij)}. Then we may compute the derivative of $\vp_r$ in the second component for $\iota=1$ by
\begin{align}\label{eq:Dphi(0,eta_(ij))}
   D_{(\lambda,[I_n])}\vp_r(0,\eta_{ij}^1) & = \tfrac{d}{dt}\big(U_{ij}^1(t)\cdot\diag(\lambda)\cdot U_{ij}^1(t)^\dagger\big)(0)
\nonumber
\\ & = 
\eta_{ij}^1\cdot\diag(\lambda) - \diag(\lambda)\cdot \eta_{ij}^1
\nonumber
\\  & = \begin{cases}
           (\lambda_j-\lambda_i)\cdot (E_{ij}+E_{ji}) & \text{if $1\leq j<i\leq r$}
        \\ \lambda_j\cdot (E_{ij}+E_{ji}) & \text{if $r+1\leq i\leq n$, $1\leq j\leq r$}\; ,
        \end{cases}
\end{align}
and similarly for $\iota\in\{\bi,\bj,\bk\}$
\begin{align}\label{eq:Dphi(0,eta_(ij)^i)}
   D_{(\lambda,[I_n])}\vp_r(0,\eta_{ij}^\iota) & = \tfrac{d}{dt}\big(U_{ij}^\iota(t)\cdot\diag(\lambda)\cdot U_{ij}^\iota(t)^\dagger\big)(0)
\nonumber
\\ & = 
\eta_{ij}^\iota\cdot\diag(\lambda) - \diag(\lambda)\cdot \eta_{ij}^\iota
\nonumber
\\  & = \begin{cases}
           (\lambda_j-\lambda_i)\cdot \iota(E_{ij}-E_{ji}) & \text{if $1\leq j<i\leq r$}
        \\ \lambda_j\cdot \iota(E_{ij}-E_{ji}) & \text{if $r+1\leq i\leq n$, $1\leq j\leq r$}\; .
        \end{cases}
\end{align}
Note that the $D_{(\lambda,[I_n])}\vp_r(0,\eta_{ij}^\iota)$ are pairwise orthogonal elements of $\Herbn$. Moreover, the directions $
\eta_{ij}^1,\eta_{ij}^\bi,\eta_{ij}^\bj,\eta_{ij}^\bk$ have length~$1$, while the directions $E_{ij}+E_{ji},\bi(E_{ij}-E_{ji}),\bj(E_{ij}-E_{ji}),\bk(E_{ij}-E_{ji})$ have lengths~$\sqrt{2}$. Taking this into account, we get
\begin{align*}
   \ndet(D_{(\lambda,[U])}\vp_r) & = \prod_{1\leq j<i\leq r} (\sqrt{2}\cdot (\lambda_j-\lambda_i))^4 \cdot \prod_{\substack{r+1\leq i \leq n\\ 1\leq j\leq r}} (\sqrt{2}\cdot \lambda_j)^4
\\ & = 2^{r(r-1)}\cdot \prod_{1\leq j<i\leq r} (\lambda_j-\lambda_i)^4 \cdot 2^{2r(n-r)}\cdot \prod_{1\leq j\leq r} \lambda_j^{4(n-r)} \; . \qedhere
\end{align*}
\end{proof}

It remains to compute the principal curvatures of the essential strata $M_{n,1^{(r)}}$ before we can use Theorem~\ref{thm:form-intr-vol} to compute the intrinsic volumes of $K_n$. We will do this in the following proposition. Note that the dimension of $M_{n,1^{(r)}}$ is given by
\begin{align*}
\dim M_{n,1^{(r)}} & = \dim P_r + \dim G_{n,r} \stackrel{\eqref{eq:dim(G(n)/G(n,rho))}}{=} r-1+\be\cdot \left(\tbinom{n}{2}-\tbinom{n-r}{2}\right)
\\ & = \be r(n-r) + r-1+\be\tbinom{r}{2} \; .
\end{align*}

\begin{proposition}\label{prop:princ-curv-M_m}
Let $A=U\cdot \diag(\lambda)\cdot U^\dagger\in M_{n,1^{(r)}}$ with $\lambda\in P_r$. Furthermore, let $A''\in \C_{\be,n-r}$, so that $B:=U\cdot \begin{pmatrix} 0 & 0 \\ 0 & -A''\end{pmatrix} \cdot U^\dagger\in N_A(K_n)$ is a vector in the normal cone of $K_n$ at $A$. If $\mu_1\geq\ldots\geq\mu_{n-r}\geq0$ denote the eigenvalues of $A''$, then the principal curvatures of~$M_{n,1^{(r)}}$ at~$A$ in direction $-B$ are given by
\begin{align*}
   \frac{\mu_1}{\lambda_1} ,\ldots,\frac{\mu_{n-r}}{\lambda_1} \,,\; \frac{\mu_1}{\lambda_2},\ldots,\frac{\mu_{n-r}}{\lambda_2} \;,\ldots,\; \frac{\mu_1}{\lambda_r},\ldots,\frac{\mu_{n-r}}{\lambda_r} \; & \text{ (each of these values $\be$-times)}\; ,
\\ 0 ,\ldots,0 \; & \text{ (overall $(r-1+\be\tbinom{r}{2})$-times)} \; .
\end{align*}
\end{proposition}

\begin{proof}
By orthogonal invariance we may assume w.l.o.g.~that $U=I_n$, so that $A=\diag(\lambda)$ and $A'' = \diag(\mu)$. From Proposition~\ref{prop:phi_(n,r)} we get that the tangent space of $M_{n,1^{(r)}}$ at $A$ is given by (omitting the argument $(\lambda,[I_n])$)
  \[ T_A M_{n,1^{(r)}} = D\vp_r\big(T_\lambda P_r\times \ol{\Skew_n}\big) \; . \]
It is easily seen that all the vectors in $D\vp_r(T_\lambda P_r\times \{0\})$ are principal directions of $M_{n,1^{(r)}}$ at~$A$ with principal curvature~$0$, thus giving $r-1$ of the claimed $r-1+\be\tbinom{r}{2}$ zero curvatures.

Concerning the second component, we again only consider the quaternion case $\be=4$, the other cases being similar. As in the proof of Proposition~\ref{prop:phi_(n,r)} let $U_{ij}^\iota\colon\IR\to G(n)$, with $\iota\in\{1,\bi,\bj,\bk\}$ and $(i,j)\in\mI$ (cf.~\eqref{eq:def-I1,I2}), be curves such that the induced curves $[U_{ij}^\iota]\colon\IR\to G_{n,r}$ define the directions~$\eta_{ij}^\iota$, cf.~\eqref{eq:def-eta_(ij)}. We denote the images of $D\vp_r$ by
\begin{equation}\label{eq:zeta_ij}
  \zeta_{ij}^\iota := D\vp_r(0,\eta_{ij}^\iota) \in T_A M_{n,1^{(r)}} \;,\quad \iota\in\{1,\bi,\bj,\bk\} \; .
\end{equation}
Explicit formulas for the vectors~$\zeta_{ij}^\iota$ are given in~\eqref{eq:Dphi(0,eta_(ij))}.

We define normal extensions of $-B=\diag(0,\mu)$ along the curves $\vp_r\big(\lambda,\big[U_{ij}^\iota(t)\big]\big)$ via
\begin{align*}
v_{ij}^\iota(t) & := U_{ij}^\iota(t)\cdot \diag(0,\mu)\cdot U_{ij}^\iota(t)^\dagger \;,\quad \iota\in\{1,\bi,\bj,\bk\} \; .
\end{align*}
Differentiating these normal extensions $t=0$ yields for $\iota=1$, using $\eta_{ij}^1=E_{ij}-E_{ji}$,
\begin{align*}
  \tfrac{d}{dt} v_{ij}^1(0) & = (E_{ij}-E_{ji})\cdot \diag(0,\mu) - \diag(0,\mu)\cdot (E_{ij}-E_{ji})
\\ & = \begin{cases}
          0 & \text{if $1\leq j<i\leq r$}
       \\ -\mu_{i-r}\cdot (E_{ij}+E_{ji}) & \text{if $r+1\leq i\leq n$, $1\leq j\leq r$}\; ,
       \end{cases}
\intertext{and for $\iota\in\{\bi,\bj,\bk\}$, using $\eta_{ij}^\iota=\iota(E_{ij}+E_{ji})$,}
  \tfrac{d}{dt} v^\iota(0) & = \iota(E_{ij}+E_{ji})\cdot \diag(0,\mu) - \diag(0,\mu)\cdot \iota(E_{ij}+E_{ji})
\\ & = \begin{cases}
          0 & \text{if $1\leq j<i\leq r$}
       \\ -\mu_{i-r}\cdot \iota(E_{ij}-E_{ji}) & \text{if $r+1\leq i\leq n$, $1\leq j\leq r$}\; .
       \end{cases}
\end{align*}
Comparing this with the values of $\zeta_{ij}^1,\zeta_{ij}^\bi,\zeta_{ij}^\bj,\zeta_{ij}^\bk$ given in~\eqref{eq:Dphi(0,eta_(ij))}, implies for $\iota\in\{1,\bi,\bj,\bk\}$
\begin{align*}
{\textstyle\frac{d}{dt}} v_{ij}^\iota(0)
      = \begin{cases}
           0\cdot \zeta_{ij}^\iota & \text{if $1\leq j<i\leq r$}
        \\ -\frac{\mu_{i-r}}{\lambda_j}\cdot \zeta_{ij}^\iota & \text{if $r+1\leq i\leq n$, $1\leq j\leq r$}\; .
        \end{cases}
\end{align*}
We conclude that the directions $\zeta_{ij}^1,\zeta_{ij}^\bi,\zeta_{ij}^\bj,\zeta_{ij}^\bk$ are principal directions with curvature~$0$ and~$\frac{\mu_{i-r}}{\lambda_j}$, respectively.
\end{proof}

Before we give the proof of Theorem~\ref{thm:main}, we state a small lemma, that will come in handy for an integral conversion later on.

\begin{lemma}\label{lem:integr-gauss-kernel}
Let $f\colon\IR^n\setminus\{0\}\to\IR$ be a homogeneous function of degree~$k$, i.e.,~$f(x)=\|x\|^k\cdot f(\|x\|^{-1}\cdot x)$. Then for a Borel set $U\subseteq S^{n-1}$ and $\hat{U}=\{s\cdot p\mid s\geq0 \,,\; p\in U\}$
  \[ \underset{p\in U}{\int} f(p)\,dp = \frac{1}{2^{\frac{n+k}{2}-1}\cdot \Gamma(\frac{n+k}{2})}\cdot \underset{x\in \hat{U}}{\int} e^{-\frac{\|x\|^2}{2}}\cdot f(x)\,dx \; . \]
\end{lemma}

\begin{proof}
The normal Jacobian of the projection $\Pi\colon\IR^n\setminus\{0\}\to S^{n-1}$, $\Pi(x)=\|x\|^{-1}\cdot x$, is given by $\ndet(D_x\Pi)=\|x\|^{-(n-1)}$. From the coarea formula (Lemma~\ref{lem:coarea}) we thus get
\begin{align*}
   \underset{x\in \hat{U}}{\int} e^{-\frac{\|x\|^2}{2}}\cdot f(x)\,dx & = \underset{p\in U}{\int} \int_0^\infty s^{n-1}\cdot e^{-s^2/2}\cdot f(s\cdot p) \,ds\,dp = \underset{p\in U}{\int} f(p)\,dp \cdot \int_0^\infty s^{n-1+k}\cdot e^{-s^2/2} \,ds .
\end{align*}
Substituting $t:=s^2/2$, and using the well-known formula $\int_0^\infty t^{z-1}\cdot e^{-t}\,dt=\Gamma(z)$, we obtain
\begin{align*}
   \int_0^\infty s^{n-1+k} \cdot e^{-s^2/2} \,ds
& = 2^{\frac{n+k}{2}-1}\cdot \int_0^\infty t^{\frac{n+k}{2}-1}\cdot e^{-t} \,dt
= 2^{\frac{n+k}{2}-1}\cdot \Gamma(\tfrac{n+k}{2}) \; . \qedhere 
\end{align*}
\end{proof}

\begin{proof}[Proof of Theorem~\ref{thm:main}]
Recall from~\eqref{eq:def-Phi_j(b,n,r)} that $\Phi_j(\be,n,r)=\Phi_j(\Cbn,\Wbnr)$ and in the stratification~\eqref{eq:dec-ID_n} of $K_n=\Cbn\cap S(\Herbn)$ only the strata $M_{n,1^{(r)}}$ are essential, cf.~Proposition~\ref{prop:strat-struct}. Denoting $\de(n,r) := \dim M_{n,1^{(r)}}+2 = \be r(n-r)+ d_{\be,r} + 1$, we thus obtain from~Theorem~\ref{thm:form-intr-vol}
  \[ \Phi_j(\be,n,r) = \frac{1}{\mO_{j-1}\cdot\mO_{d_{\be,n}-j-1}} \underset{A\in M_{n,1^{(r)}}}{\int}\;\underset{B\in N_A^S}{\int} \sigma_{\de(n,r)-j-1}^{(r)}(A,-B) \;dN_A^S\;dM_{n,1^{(r)}} \; , \]
where the superscript in $\sigma_{\de(n,r)-j-1}^{(r)}$ shall indicate the dependence on $M_{n,1^{(r)}}$. 

Recall that in Proposition~\ref{prop:princ-curv-M_m} we have computed the principal curvatures of $M_{n,1^{(r)}}$. Using the notation $(x\otimes y^{-1})^{\times \be} = \big(x\otimes y^{-1},\ldots,x\otimes y^{-1}\big)$ ($\be$-times) and $x\otimes y^{-1}$ defined in~\eqref{eq:xotimesy}, we obtain
\begin{equation}\label{eq:Phi_j(b,n,r)-withcurv}
  \Phi_j(\be,n,r) = \frac{1}{\mO_{j-1}\cdot\mO_{d_{\be,n}-j-1}} \underset{A\in M_{n,1^{(r)}}}{\int}\;\underset{B\in N_A^S}{\int} \sigma_{\de(n,r)-j-1}\left((\lambda^{-1}\otimes\mu)^{\times\be}\right) \,dN_A^S\,dM_{n,1^{(r)}} ,
\end{equation}
where $\lambda$ and $\mu$ denote the (positive) eigenvalues of $A$ and $-B$, respectively. Using the relation $\sigma_k(\frac{1}{x_1},\ldots,\frac{1}{x_N})=(x_1\cdots x_N)^{-N}\cdot \sigma_{N-k}(x_1,\ldots,x_N)$ and observing $\de(n,r)-j-1 = \be r(n-r) + d_{\be,r} - j$, we can rewrite the integrand via
\begin{equation}\label{eq:rewr-integr}
  \sigma_{\de(n,r)-j-1}\left((\lambda^{-1}\otimes\mu)^{\times\be}\right) = \sigma_{j-d_{\be,r}}\left((\lambda\otimes\mu^{-1})^{\times\be}\right)\cdot \frac{\prod_{i=1}^{n-r} \mu_i^{\be r}}{\prod_{i=1}^r \lambda_i^{\be(n-r)}} \; .
\end{equation}

In Proposition~\ref{prop:phi_(n,r)} we have shown that the Normal Jacobian of $\vp_n$ at $(\mu,[U])$ is given by $2^{n(n-1) \be/4}\cdot \Delta(\mu)^\be$. The normal cone of~$\Cbn$ at $A\in M_{n,1^{(r)}}$ is isometric to $\C_{\be,n-r}$, cf.~Proposition~\ref{prop:faces-SDP-cone}. Further, $M_{n-r,1^{(n-r)}}$ equals $K_{n-r} = \C_{\be,n-r} \cap S(\Her_{\be,n-r})$ up to strata of lower dimension. Taking into account the Normal Jacobian of $\varphi_{n-r}$ determined in Proposition~\ref{prop:phi_(n,r)} (note that~$n$ needs to be replaced by~$n-r$), we can transform the inner integral of~\eqref{eq:Phi_j(b,n,r)-withcurv} via the coarea formula to obtain
  \[ \underset{B\in N_A^S}{\int} f(\lambda,\mu) \,dN_A^S = \underset{P_{n-r}\times G_{n-r,n-r}}{\int} f(\lambda,\mu)\cdot 2^{(n-r)(n-r-1) \be/4}\cdot \Delta(\mu)^\be \,d(\mu,[U_2]) \; , \]
where we have abbreviated $f(\lambda,\mu)$ for the integrand~\eqref{eq:rewr-integr}.

Similarly, we may transform the outer integral of~\eqref{eq:Phi_j(b,n,r)-withcurv} by applying the coarea formula to the map~$\vp_r$. As a result we obtain
\begin{align}\label{eq:integr-transfooorm}
   \underset{A\in M_{n,1^{(r)}}}{\int} & \;\underset{B\in N_A^S}{\int} f(\lambda,\mu) \;dN_A^S\;dM_{n,1^{(r)}} = \underset{(\lambda,[U_1])\in P_r\times G_{n,r}}{\int}\;\; \underset{(\mu,[U_2])\in P_{n-r}\times G_{n-r,n-r}}{\int} f(\lambda,\mu) \nonumber
\\ & \cdot 2^{r(2n-r-1) \be/4}\cdot \Delta(\lambda)^\be \cdot \prod_{i=1}^r \lambda_i^{\be(n-r)}\cdot 2^{(n-r)(n-r-1) \be/4}\cdot \Delta(\mu)^\be \;d(\mu,[U_2]) \,d(\lambda,[U_1]) \; .
\end{align}
Note that we have
  \[ \vol G_{n,r}\cdot \vol G_{n-r,n-r} \stackrel{\eqref{eq:vol(G_(n,r))}}{=} \frac{\vol G(n)}{\mO_{\be-1}^r\cdot \vol G(n-r)} \cdot \frac{\vol G(n-r)}{\mO_{\be-1}^{n-r}} = \frac{\vol G(n)}{\mO_{\be-1}^n} \; . \]
Replacing $f(\lambda,\mu)$ again by~\eqref{eq:rewr-integr}, the integral~\eqref{eq:integr-transfooorm} simplifies to
\begin{align}
   & 2^{n(n-1)\be/4} \cdot \frac{\vol G(n)}{\mO_{\be-1}^n} \cdot \underset{P_r}{\int} \;\underset{P_{n-r}}{\int} \Delta(\lambda)^\be \cdot \Delta(\mu)^\be\cdot \sigma_{j-d_{\be,r}}\left((\lambda\otimes\mu^{-1})^{\times \be}\right) \cdot \prod_{i=1}^{n-r} \mu_i^{\be r} \;d\lambda\,d\mu
\nonumber
\\ & \stackrel{(*)}{=} \frac{(2\pi)^{n(n-1)\be/4}\cdot n!}{F_n(\be/2)}\cdot \underset{P_r}{\int} \;\underset{P_{n-r}}{\int} \Delta(\lambda)^\be \cdot \Delta(\mu)^\be\cdot f_{\be,j-d_{\be,r}}(\lambda;\mu) \;d\lambda\,d\mu \; ,
\label{eq:integr_spher}
\end{align}
where in~$(*)$ we have used~\eqref{eq:f_(b,k)(x;y)-expl} and the small computation
  \[ \frac{\vol G(n)}{\mO_{\be-1}^n} \stackrel{\eqref{eq:vol(G(n))}}{=} \frac{2^n\cdot \pi^{n(n+1)\be/4}\cdot \prod_{i=1}^n\frac{1}{\Gamma(\frac{\be i}{2})}}{(2\cdot \pi^{\be/2} / \Gamma(\frac{\be}{2}))^n} = \pi^{n(n-1)\be/4}\cdot \prod_{i=1}^n\frac{\Gamma(\frac{\be}{2})}{\Gamma(\frac{\be i}{2})} \stackrel{\eqref{eq:F_n(b/2)}}{=} \frac{\pi^{n(n-1)\be/4}\cdot n!}{F_n(\be/2)} \; . \]
The integrand in~\eqref{eq:integr_spher} is bihomogeneous in~$\lambda$ and~$\mu$. Its degree in $\lambda$ equals $\be\tbinom{r}{2} + j - d_{\be,r} = j-r$, and its degree in $\mu$ is given by $\be\tbinom{n-r}{2}+\be r(n-r) - j + d_{\be,r} = \be\tbinom{n}{2}-j+r$. Using Lemma~\ref{lem:integr-gauss-kernel} twice, we get
\begin{align*}
   \eqref{eq:integr_spher} & = \frac{(2\pi)^{n(n-1)\be/4}\cdot n!}{F_n(\be/2)}\cdot \frac{2}{2^{j/2}\cdot \Gamma(\frac{j}{2})}\cdot \frac{2}{2^{\left(\be\binom{n}{2}-j+n\right)/2}\cdot \Gamma\big(\frac{\be\binom{n}{2}+n-j}{2}\big)}
\\ & \hspace{2cm} \cdot \underset{\hat{P}_r}{\int} \; \underset{\hat{P}_{n-r}}{\int} e^{-\frac{\|\lambda\|^2+\|\mu\|^2}{2}}\cdot \Delta(\lambda)^\be \cdot \Delta(\mu)^\be\cdot f_{\be,j-d_{\be,r}}(\lambda;\mu) \;d\lambda\,d\mu
\\ & = \frac{\mO_{j-1}\cdot\mO_{d_{\be,n}-j-1}\cdot n!}{F_n(\be/2)\cdot (2\pi)^{n/2}}\cdot \underset{\hat{P}_r}{\int} \; \underset{\hat{P}_{n-r}}{\int} e^{-\frac{\|\lambda\|^2+\|\mu\|^2}{2}}\cdot \Delta(\lambda)^\be \cdot \Delta(\mu)^\be\cdot f_{\be,j-d_{\be,r}}(\lambda;\mu) \;d\lambda\,d\mu \; .
\end{align*}
The positive orthant $\IR_+^r$ decomposes into $r!$ isometric copies of $\hat{P}_r$, such that their interiors are disjoint. More precisely, the copies of $\hat{P}_r$ are parametrized by the permutations of $\{1,\ldots,r\}$, which indicate the order of the components of a vector in~$\IR_+^r$. The same applies to $\IR_+^{n-r}$ and~$\hat{P}_{n-r}$. As the Vandemonde determinant is antisymmetric, and $f_{\be,k}(\lambda;\mu)$ is symmetric both in~$\lambda$ and in~$\mu$, we finally obtain
\begin{align*}
   \Phi_j(\be,n,r) & = \frac{n!}{F_n(\be/2)\cdot (2\pi)^{n/2}}\cdot \underset{\hat{P}_r}{\int} \; \underset{\hat{P}_{n-r}}{\int} e^{-\frac{\|\lambda\|^2+\|\mu\|^2}{2}}\cdot \Delta(\lambda)^\be \cdot \Delta(\mu)^\be\cdot f_{\be,j-d_{\be,r}}(\lambda;\mu) \;d\lambda\,d\mu
\\ & = \binom{n}{r}\cdot \frac{1}{F_n(\be/2)\cdot (2\pi)^{n/2}}\cdot \underset{\nu:=(\lambda,\mu)\in\IR_+^n}{\int} e^{-\frac{\|\nu\|^2}{2}}\cdot |\Delta(\lambda)|^\be \cdot |\Delta(\mu)|^\be\cdot f_{\be,j-d_{\be,r}}(\lambda;\mu) \;d\nu
\\ & \stackrel{\eqref{eq:def-J}}{=} \binom{n}{r}\cdot \frac{J_\be(n,r,j-d_{\be,r})}{F_n(\be/2)} \; . \qedhere
\end{align*}
\end{proof}

\appendix

\section{Support measures}\label{sec:supp-meas}

The kinematic formula, as we state it in~\eqref{eq:random-proj} and which we need in the proof of Theorem~\ref{thm:CP}, follows, as well as the other formulas in Theorem~\ref{thm:kinem-form}, from a general kinematic formula for \emph{support measures} that Glasauer has proved in~\cite{Gl}, cf.~also~\cite{Gl:Summ,HS:02}. As this general formula is not easy to trace in the literature, we will give in this section a short account of support measures, and we will derive Theorem~\ref{thm:kinem-form} from Glasauer's kinematic formula for support measures.

\subsection{Definition and properties}

A major drawback of the curvature measures is that while they localize the notion of intrinsic volumes, they do not possess a duality structure. In fact, one may interpret the curvature measures as primal localizations of the intrinsic volumes, while the support measures may be thought of as primal-dual localizations of the intrinsic volumes. They generalize the curvature measures, and they do admit a duality structure. To obtain this duality we need to replace the Borel algebra $\hmB(\E)$, defined in~\eqref{eq:def-hat(B)(E)} by another $\sigma$-algebra, which we define next.

\begin{definition}
Let $\E,\E'$ be finite-dimensional euclidean spaces, and let $\mB(\E\times\E')$ denote the Borel $\sigma$-algebra on $\E\times\E'$. Then we call
\begin{equation}\label{eq:def-hat(B)(E,E')}
  \hmB(\E,\E') := \{ \M\in\mB(\E\times\E')\mid \forall \lambda,\mu>0: (\lambda,\mu) \M=\M\} \; ,
\end{equation}
where $(\lambda,\mu) \M:=\{(\lambda x,\mu v)\mid (x,v)\in \M, x\in\E,v\in\E'\}$, the \emph{biconic (Borel) $\sigma$-algebra on~$(\E,\E')$}. If $\E'=\E$ then we call $\hmB(\E,\E)$ the \emph{biconic (Borel) $\sigma$-algebra on $\E$}.
\end{definition}

It is straightforward to show that the biconic $\sigma$-algebra $\hmB(\E,\E')$ satisfies the axioms of a $\sigma$-algebra.
In the following we will only consider the case $\E'=\E$. In this case we have the natural involution $.^*\colon \E\times\E\to\E\times\E$, $(x,v)^*:=(v,x)$. This map induces an involution $\M\mapsto\M^*:=\{(v,x)\mid (x,v)\in\M\}$ on $\hmB(\E,\E)$. The group of isometries of~$\E$, i.e., the orthogonal group~$O(\dim\E)$ having fixed an orthonormal basis in~$\E$, acts on~$\hmB(\E,\E)$ via $Q\M:=\{(Qx,Qv)\mid (x,v)\in\M\}$. 

The \emph{support measures of $C$} are measures on the biconic $\sigma$-algebra on $\E$
  \[ \Theta_0(C,.),\ldots,\Theta_d(C,.)\colon \hmB(\E,\E)\to\IR_+ \; . \]
In the following we will only give the definition of the support measures for polyhedral cones in analogy to~\eqref{eq:V_j(C)-polyhdrl} and~\eqref{eq:Phi_j(C,M)-polyhdrl}. By continuity one obtains from this also a definition for support measures for general closed convex cones. More precisely, it can be shown, cf.~\cite{Gl}, that~$\Theta_j$ can be extended to the set of all closed convex cones in such a way that if the sequence $C_1,C_2,\ldots$ converges in the Hausdorff metric to~$C$, then $\lim_{k\to\infty} \Theta_j(C_k,\M)=\Theta_j(C,\M)$ for all $\M\in\hmB(\E,\E)$. We skip a definition via volumes of tubes similar to Definition~\ref{def:Phi_j(C,M),V_j(C)} as we will not need such a characterization.

Let $C\subseteq\E$ be a polyhedral cone in $d$-dimensional euclidean space~$\E$. It is well-known, cf.~for example~\cite[\S2]{Z:95} or~\cite[\S14]{G:07}, that the set of faces of a polyhedral cone forms a lattice. Furthermore, the face lattice of the dual $\breve{C}$ is the opposite lattice of the primal cone~$C$, cf.~\cite[\S2.3]{Z:95}. More precisely, if $\bar{F}\subseteq C$ is a face of~$C$ and if $F:=\relint(\bar{F})$ denotes the relative interior of~$\bar{F}$, then we denote the corresponding face of the dual cone~$\breve{C}$ and its relative interior by
  \[ \bar{F}^\diamond := \spa(\bar{F})^\bot\cap\breve{C} \;,\qquad F^\diamond := \relint(\bar{F}^\diamond) \; . \]
Let $\F_j$ denote the set of the relative interiors of the $j$-dimensional faces of $C$.
If $\Pi_C$ and $\Pi_{\breve{C}}$ denote the canonical projections on~$C$ and on~$\breve{C}$, respectively, then the support measures of~$C$ evaluated in $\M\in\hmB(\E,\E)$ are given by
\begin{equation}\label{eq:Theta_j(C,M)-polyhdrl}
   \Theta_j(C,\M) = \sum_{F\in\F_j} \; \underset{x\in\N(\E)}{\Prob}\Big[(\Pi_C(x),\Pi_{\breve{C}}(x))\in (F\times F^\diamond)\cap \M\Big] \;,\quad j=0,1,\ldots,d \; .
\end{equation}
For $j>d$ we define $\Theta_j(C,\M):=0$.

Analogous to Proposition~\ref{prop:facts-intrvol} we formulate in the following proposition some well-known facts about the support measures, which may be verified using the characterization of~$\Theta_j$ in~\eqref{eq:Theta_j(C,M)-polyhdrl}.

\begin{proposition}\label{prop:facts-suppmeas}
\begin{enumerate}
  \item \label{enum:Theta-Phi}
        The support measure $\Theta_j(C,.)$ is concentrated on~$C\times\breve{C}$, that is, $\Theta_j(C,\M)=\Theta_j(C,\M\cap (C\times\breve{C}))$. Furthermore, we have $\Theta_j(C,M\times\breve{C})=\Phi_j(C,M)$ for $M\in\hmB(\E)$.
  \item The support measures are invariant under orthogonal transformations, i.e., for $Q\in O(d)$ we have $\Theta_j(QC,Q\M) = \Theta_j(C,\M)$.
  \item \label{enum:duality-Theta}
        For the support measures of the dual cone we have $\Theta_j(C,\M) = \Theta_{d-j}(\breve{C},\M^*)$. \qed
\end{enumerate}
\end{proposition}

\subsection{Lattice structures}

In this section we further examine the structure of the set of closed convex cones and of the biconic $\sigma$-algebra. This will provide a formal framework for the important duality structure, which is a specific property of spherical convex geometry.

Let the set of closed convex cones in euclidean space $\E$ be denoted by
  \[ \mC(\E) := \{C\subseteq\E\mid C \text{ closed convex cone}\} \; , \]
and let the operations $\neg,\wedge,\vee$ be defined on $\mC(\E)$ via
  \[ \neg C :=\breve{C} \;,\qquad C_1\wedge C_2 := C_1\cap C_2 \;,\qquad C_1\vee C_2 := C_1 + C_2 = \cone(C_1\cup C_2) \; , \]
where $\cone M := \{\lambda x+\mu y\mid \lambda,\mu\geq0,x,y\in M\}$ for $M\subseteq\E$. It is easily checked that the structure $(\mC(\E),\wedge,\vee,\E,\{0\},\neg)$ satisfies the axioms of an \emph{orthocomplemented lattice}, cf.~\cite{Birk:79}, where the elements $\E$ and $\{0\}$ are the neutral elements w.r.t.~the operations $\wedge$ and $\vee$, respectively. In~particular, we have the de Morgan's laws
\begin{equation}\label{eq:deMorg-C(E)}
  \dual(C_1\cap C_2) = \breve{C}_1 + \breve{C}_2 \;,\qquad \dual(C_1+C_2) = \breve{C}_1\cap \breve{C}_2 \; .
\end{equation}

The projection on linear subspaces of $\E$ fits in this framework in the following way. Let $C\in\mC(\E)$ and let $W\subseteq\E$ be a subspace of codimension~$m$. Then we have
\begin{equation}\label{eq:inters-proj}
  \dual(\breve{C}\cap W) = \neg(\neg C\wedge W) = C\vee \neg W = C+W^\bot \cong \Pi_W(C)\times \IR^m \; .
\end{equation}
So projection on linear subspaces is dual to intersection in a natural way.

The biconic $\sigma$-algebra~$\hmB(\E,\E')$ has far less structure than~$\mC(\E)$. Nevertheless, we may define operations $\wedge$ and $\vee$ on~$\hmB(\E,\E')$ in the following way. For $\M\in\hmB(\E,\E')$, $x\in\E$, $v\in\E'$, we introduce the fibers
  \[ \M_x := \Pi_1^{-1}(x)\cap \M \;,\qquad \M^v := \Pi_2^{-1}(v)\cap \M \; , \]
where $\Pi_1\colon\E\times\E'\to\E$ and $\Pi_2\colon\E\times\E'\to\E'$ denote the canonical projections onto the first and the second component, respectively.
We define the operations $\wedge$ and $\vee$ on~$\hmB(\E,\E')$ via
  \[ \M\wedge\N := \bigcup_{x\in\E} \Big(\{x\}\times (\M_x+\N_x)\Big) \;,\qquad \M\vee\N := \bigcup_{v\in\E'} \Big((\M^v+\N^v)\times \{v\}\Big) \; , \]
for $\M,\N\in\hmB(\E,\E')$.
Note that for products $\M=M_1\times M_2,\, \N=N_1\times N_2\in\hmB(\E,\E')$ we have
\begin{align}
   (M_1\times M_2)\wedge (N_1\times N_2) & = (M_1\cap N_1)\times (M_2 + N_2) \; ,
\label{eq:(M1xM2)wedge(N1xN2)=...}
\\ (M_1\times M_2)\vee (N_1\times N_2) & = (M_1 + N_1)\times (M_2\cap N_2) \; .
\label{eq:(M1xM2)vee(N1xN2)=...}
\end{align}
The following identities are readily checked for $\M,\N,\mO\in\hmB(\E,\E')$:
\begin{align*}
   \M\wedge\N & = \N\wedge\M \;, & \M\wedge(\N\wedge\mO) & = (\M\wedge\N)\wedge\mO \; ,
\\ \M\vee\N & = \N\vee\M \;, & \M\vee(\N\vee\mO) & = (\M\vee\N)\vee\mO \; .
\end{align*}
Furthermore, the elements $\E\times\{0\}$ and $\{0\}\times\E'$ are neutral w.r.t.~the operations $\wedge$ and $\vee$, i.e.,
  \[ \M\,\wedge\, (\E\times\{0\}) = \M \;,\qquad \M\,\vee\, (\{0\}\times\E') = \M \; . \]
In the case $\E=\E'$, to which we will restrict ourselves from now on, we also have an involution, which we shall denote by $\neg$, given by
  \[ \neg\M := \M^* = \{(v,x)\mid (x,v)\in\M\} \; . \]
With this involution we also have identities resembling de Morgan's laws
\begin{equation}\label{eq:deMorg-B(E,E)}
  (\M\wedge\N)^* = \M^*\vee\N^* \;,\qquad (\M\vee\N)^* = \M^*\wedge\N^* \; .
\end{equation}
Unfortunately, $\hmB(\E,\E)$ is \emph{not} a lattice with respect to the operations~$\wedge$ and~$\vee$, the reason being
  \[ \M\wedge\M = \bigcup_{x\in\E} \Big( \{x\}\times (\M_x+\M_x)\Big) \stackrel{\text{in general}}{\neq} \bigcup_{x\in\E} \Big( \{x\}\times \M_x\Big) = \M \; , \]
as the fiber $\M_x$ need not be convex.

For a cone $C\in\mC(\E)$ and a biconic set $\M\in\mB(\E,\E)$, we abbreviate $\neg(C,\M):=(\neg C,\neg \M)=(\breve{C},\M^*)$. In particular, we may rewrite the property of $\Theta_j$ stated in~Proposition~\ref{prop:facts-suppmeas}\eqref{enum:duality-Theta} in the form
\begin{equation}\label{eq:duality-Theta-nice}
  \Theta_j(C,\M) = \Theta_{d-j}(\neg(C,\M)) \; .
\end{equation}
Furthermore, we abbriviate $(C,\M)\wedge(D,\N):=(C\wedge D,\M\wedge\N)$ and $(C,\M)\vee(D,\N):=(C\vee D,\M\vee\N)$.

\subsection{Glasauer's kinematic formula for support measures}\label{sec:kinem-form-supp-meas}


We let the orthogonal group $O(d)$ operate on pairs $(\M,\N)$ of biconic sets $\M,\N\in\hmB(\E,\E)$ via $Q(\M,\N):=(Q\M,Q\N)$ for $Q\in O(d)$.
The following theorem is Satz~6.1.1/6.1.2 in~\cite{Gl} and Theorem~9/10 in~\cite{Gl:Summ}, cf.~also~\cite[\S2.4]{HS:02}.

\begin{thm}[Glasauer]\label{thm:Glas-kinem}
Let $\E$ be a $d$-dimensional euclidean space, let $C,D\in\mC(\E)$, and let $\M,\N\in\hmB(\E,\E)$ be such that $\M\subseteq C\times\breve{C}$ and $\N\subseteq D\times\breve{D}$. Then for uniformly random $Q\in O(d)$ and $1\leq j\leq d-1$,
\begin{align}
  \underset{Q}{\IE}\Big[ \Theta_j\big((C,\M)\wedge Q(D,\N)\big) \Big] & = \sum_{k=j+1}^{d-1} \Theta_k(C,\M) \cdot \Theta_{d+j-k}(D,\N) \; ,
\label{eq:Glas-kinem1}
\\\underset{Q}{\IE}\Big[ \Theta_j\big((C,\M)\vee Q(D,\N)\big) \Big] & = \sum_{k=1}^{j-1} \Theta_k(C,\M) \cdot \Theta_{j-k}(D,\N) \; .
\label{eq:Glas-kinem2}
\end{align}
\end{thm}

\begin{remark}
The two formulas~\eqref{eq:Glas-kinem1} and~\eqref{eq:Glas-kinem2} are in fact equivalent, as one follows from the other via duality. For example, assuming~\eqref{eq:Glas-kinem1}, we obtain
\begin{align*}
   & \underset{Q}{\IE}\Big[ \Theta_j\big((C,\M)\vee Q(D,\N)\big) \Big] \stackrel{\eqref{eq:duality-Theta-nice}}{=} \underset{Q}{\IE}\Big[ \Theta_{d-j}\big(\neg((C,\M)\vee Q(D,\N))\big) \Big]
\\ & \hspace{-3mm}\stackrel{\eqref{eq:deMorg-C(E)},\eqref{eq:deMorg-B(E,E)}}{=} \underset{Q}{\IE}\Big[ \Theta_{d-j}\big(\neg (C,\M)\wedge Q(\neg (D,\N))\big) \Big] \stackrel{\eqref{eq:Glas-kinem1}}{=} \sum_{k=d-j+1}^{d-1} \Theta_k(\neg (C,\M)) \cdot \Theta_{2d-j-k}(\neg (D,\N))
\\ & \stackrel{\eqref{eq:duality-Theta-nice}}{=} \sum_{k=d-j+1}^{d-1} \Theta_{d-k}(C,\M) \cdot \Theta_{k+j-d}(D,\N) \stackrel{[\ell:=d-k]}{=} \sum_{\ell=1}^{j-1} \Theta_\ell(C,\M) \cdot \Theta_{j-\ell}(D,\N) \; .
\end{align*}
\end{remark}

We finish this section by deriving Theorem~\ref{thm:kinem-form} from Theorem~\ref{thm:Glas-kinem}.

\begin{proof}[Proof of Theorem~\ref{thm:kinem-form}]
Let $C\subseteq\IR^d$ be a closed convex cone, and let $W\subseteq\IR^d$ be a uniformly random subspace of codimension~$m\in\{1,\ldots,d-1\}$. Then for $M\in\hmB(\IR^d)$ with $M\subseteq C$ and $1\leq j\leq d-m$ we have
\begin{align*}
   \Phi_j(C\cap W & ,M\cap W) \stackrel{\text{Prop.~\ref{prop:facts-suppmeas}\eqref{enum:Theta-Phi}}}{=} \Theta_j(C\cap W,(M\cap W)\times\dual(C\cap W))
\\ & \stackrel{\eqref{eq:deMorg-C(E)}}{=} \Theta_j(C\cap W,(M\cap W) \times (\breve{C} + W^\bot))
\stackrel{\eqref{eq:(M1xM2)wedge(N1xN2)=...}}{=} \Theta_j(C\cap W,(M\times\breve{C})\wedge (W\times W^\bot))
\\ & = \Theta_j( (C,M\times\breve{C})\wedge (W,W\times W^\bot)) \; .
\end{align*}
Writing $W=QW_0$ for some fixed linear subspace $W_0\subseteq\IR^d$ of codimension~$m$ and with $Q\in O(d)$ uniformly at random, we thus obtain
\begin{align*}
   \underset{W}{\IE}\big[\Phi_j(C\cap W,M)\big] & = \underset{Q}{\IE}\big[\Theta_j( (C,M\times \breve{C})\wedge Q(W_0,W_0\times W_0^\bot))\big]
\\ & \stackrel{\eqref{eq:Glas-kinem1}}{=} \sum_{k=j+1}^{d-1} \underbrace{\Theta_k(C,M\times \breve{C})}_{=\Phi_k(C,M)}\cdot \underbrace{\Theta_{d+j-k}(W_0,W_0\times W_0^\bot)}_{=V_{d+j-k}(\IR^{d-m})\stackrel{\text{Prop.~\ref{prop:facts-intrvol}\eqref{enum:ivol-sum=1}}}{=}\delta_{k,j+m}} = \Phi_{j+m}(C,M) \; ,
\end{align*}
which is~\eqref{eq:random-inters}. This implies~\eqref{eq:random-inters-2} via
\begin{align*}
   \underset{W}{\IE}\big[V_0(C\cap W)\big] & \stackrel{\text{Prop.~\ref{prop:facts-intrvol}\eqref{enum:ivol-sum=1}}}{=} \underset{W}{\IE}\bigg[1-\sum_{j=1}^{d-m}V_j(C\cap W)\bigg] = 1-\sum_{j=1}^{d-m}\underset{W}{\IE}\big[V_j(C\cap W)\big]
\\ & \stackrel{\eqref{eq:random-inters}}{=} 1-\sum_{j=1}^{d-m} V_{j+m}(C) \stackrel{\text{Prop.~\ref{prop:facts-intrvol}\eqref{enum:ivol-sum=1}}}{=} V_0(C)+V_1(C)+\ldots+V_m(C) \; .
\end{align*}

As for the projection formula, we have
\begin{align*}
   & \Phi_j(\Pi_W(C),\Pi_W(M)) \stackrel{\text{Prop.~\ref{prop:facts-intrvol}\eqref{enum:Phi_j(Pi_W(C),Pi_W(M))=...}}}{=} \Phi_{j+m}(C+W^\bot,M+W^\bot)
\\ & \stackrel{\text{Prop.~\ref{prop:facts-suppmeas}\eqref{enum:Theta-Phi}}}{=} \Theta_{j+m}(C+W^\bot,(M+W^\bot)\times \dual(C+W^\bot))
\\ & \stackrel{\eqref{eq:deMorg-C(E)}}{=} \Theta_{j+m}(C+W^\bot,(M+W^\bot)\times (\breve{C}\cap W))
\\ & \stackrel{\eqref{eq:(M1xM2)vee(N1xN2)=...}}{=} \Theta_{j+m}( C+ W^\bot,(M\times\breve{C})\vee (W^\bot,W^\bot\times W)) = \Theta_{j+m}( (C,M\times\breve{C})\vee (W^\bot,W^\bot\times W)) \; .
\end{align*}
It thus follows that
\begin{align*}
   \underset{W}{\IE}\big[\Phi_j(\Pi_W(C),\Pi_W(M))\big] & = \underset{Q}{\IE}\big[\Theta_{j+m}( (C,M\times\breve{C})\vee Q(W_0^\bot,W_0^\bot\times W_0))\big]
\\ & \stackrel{\eqref{eq:Glas-kinem2}}{=} \sum_{k=1}^{j+m-1} \underbrace{\Theta_k(C,M\times \breve{C})}_{=\Phi_k(C,M)}\cdot \underbrace{\Theta_{j+m-k}(W_0^\bot,W_0^\bot\times W_0)}_{=V_{j+m-k}(\IR^m)=\delta_{k,j}} = \Phi_j(C,M) \; ,
\end{align*}
which is~\eqref{eq:random-proj}. Analogously to the above computation we obtain~\eqref{eq:random-proj-2}.
\end{proof}


\bibliography{symmetric}

\end{document}